\documentclass[a4paper,11pt]{article}
\usepackage{amsmath, amsfonts, amscd, amssymb, amsthm, enumerate, xypic}

\def\bfB{\mathbf{B}}

\newcommand{\Mat}{\operatorname{M}}
\newcommand{\End}{\operatorname{End}}
\newcommand{\car}{\chi}

\newcommand{\id}{\operatorname{id}}
\newcommand{\GL}{\operatorname{GL}}
\newcommand{\Ker}{\operatorname{Ker}}
\newcommand{\Irr}{\operatorname{Irr}}

\newcommand{\Vect}{\operatorname{span}}
\newcommand{\im}{\operatorname{Im}}
\newcommand{\tr}{\operatorname{tr}}

\newcommand{\Co}{\operatorname{Co}}

\newcommand{\Tor}{\operatorname{Tor}}
\newcommand{\Span}{\operatorname{Span}}

\renewcommand{\setminus}{\smallsetminus}

\def\F{\mathbb{F}}
\def\K{\mathbb{K}}

\def\N{\mathbb{N}}
\def\Z{\mathbb{Z}}

\def\calA{\mathcal{A}}
\def\calB{\mathcal{B}}

\def\calM{\mathcal{M}}

\def\calR{\mathcal{R}}

\def\lcro{\mathopen{[\![}}
\def\rcro{\mathclose{]\!]}}

\theoremstyle{definition}
\newtheorem{Def}{Definition}[section]
\newtheorem{Not}[Def]{Notation}

\theoremstyle{plain}
\newtheorem{theo}{Theorem}[section]
\newtheorem{prop}[theo]{Proposition}
\newtheorem{cor}[theo]{Corollary}
\newtheorem{lemma}[theo]{Lemma}

\theoremstyle{plain}
\newtheorem{conj}{Conjecture}

\theoremstyle{remark}
\newtheorem{Rems}{Remarks}
\newtheorem{Rem}[Rems]{Remark}

\title{Sums and products of two quadratic endomorphisms of a countable-dimensional vector space}
\author{Cl\'ement de Seguins Pazzis\footnote{Universit\'e de Versailles Saint-Quentin-en-Yvelines, Laboratoire de Math\'ematiques
de Versailles, 45 avenue des Etats-Unis, 78035 Versailles cedex, France}
\footnote{e-mail address: dsp.prof@gmail.com}}

\begin{document}

\thispagestyle{plain}

\maketitle

\begin{abstract}
Let $V$ be a vector space with countable dimension over a field, and let
$u$ be an endomorphism of it which is locally finite, i.e.\ $(u^k(x))_{k \geq 0}$ is linearly dependent for all $x$ in $V$.
We give several necessary and sufficient conditions for the decomposability of $u$ into the sum of two square-zero endomorphisms.
Moreover, if $u$ is invertible, we give necessary and sufficient conditions for the decomposability of $u$
into the product of two involutions, as well as for the decomposability of $u$ into the product of two unipotent endomorphisms of index $2$.
Our results essentially extend the ones that are known in the finite-dimensional setting.

In particular, we obtain that every strictly upper-triangular infinite matrix with entries in a field is the sum of two square-zero infinite matrices
(potentially non-triangular, though), and that every upper-triangular infinite matrix (with entries in a field) with only $\pm 1$ on the diagonal
is the product of two involutory infinite matrices.
\end{abstract}



\section{Introduction}

\subsection{The problem}\label{problemSection}

Throughout, $\F$ denotes an arbitrary field whose characteristic we denote by $\chi(\F)$. Given a non-negative integer $n$, we denote by $\Mat_n(\F)$ the algebra of all $n$ by $n$ matrices with entries in $\F$. Given a vector space $V$ over $\F$, we denote by $\End(V)$
the algebra of all endomorphisms of $V$, and by $\GL(V)$ its group of invertible elements. Following the French convention, we denote by $\N$ the set of all non-negative integers.

Let $\calA$ be an $\F$-algebra.
An element $a \in \calA$ is called \textbf{quadratic} when it is annihilated by a polynomial of $\F[t]$ with degree $2$.
Among such elements are the \textbf{square-zero elements} ($a^2=0_\calA$), the \textbf{involutions}
($a^2=1_\calA$), the \textbf{idempotent elements} ($a^2=a$) and the \textbf{unipotent elements of index $2$} ($(a-1_\calA)^2=0_\calA$).

Let $p_1,\dots,p_n$ be polynomials with degree $2$ in $\F[t]$. A $(p_1,\dots,p_n)$-sum (respectively, a $(p_1,\dots,p_n)$-product) in $\calA$
is an element $x \in \calA$ that splits into $x=\sum_{k=1}^n a_k$
(respectively, $x=\prod_{k=1}^n a_k$) where $a_k \in \calA$ and $p_k(a_k)=0$ for all $k \in \lcro 1,n\rcro$.
Note that if $x$ is a $(p_1,\dots,p_n)$-sum (respectively, a $(p_1,\dots,p_n)$-product)
then every element of $\calA$ that is conjugated to $x$ is also a $(p_1,\dots,p_n)$-sum (respectively, a $(p_1,\dots,p_n)$-product).

In general, determining which elements of $\calA$ are $(p_1,\dots,p_n)$-sums or $(p_1,\dots,p_n)$-products is an intractable problem.
However, in matrix algebras, the solution is known in several interesting cases:
\begin{itemize}
\item For $p_1=\cdots=p_n=t^2-t$, the $(p_1,\dots,p_n)$-products (i.e.\ the products of $n$ idempotents) in $\Mat_N(\F)$ have been determined by Ballantine \cite{Ballantine};
\item For $p_1=\cdots=p_n=t^2-t$, the $(p_1,\dots,p_n)$-products (i.e.\ the products of $n$ involutions) in $\Mat_N(\F)$ have been determined by Gustafson, Halmos and Radjavi \cite{Gustafsonetal} for $n \geq 4$, and Djokovi\'c, Wonenburger, Hoffman and Paige
    \cite{Djokovic,HoffmanPaige,Wonenburger} for $n=2$;
\item For $p_1=\cdots=p_n=t^2$ and $n \neq 3$, the $(p_1,\dots,p_n)$-sums (i.e.\ the sums of $n$ square-zero matrices) in $\Mat_N(\F)$ have been determined by Wang, Wu and Botha \cite{WangWu,Bothasquarezero}.
\end{itemize}
We also point to a recent classification in $\Mat_N(\F)$ of all $(p_1,p_2)$-sums, and also to the one of all $(p_1,p_2)$-products \cite{dSPsumprod2}
when $p_1(0)p_2(0) \neq 0$.

In the second half of the twentieth century, there have been comparable studies on the algebra of bounded operators of a Hilbert space (see \cite{HalmosKakutani,PearcyTopping,Radjavi,WangWu}). However, it is only very recently that serious results have been discovered for the algebra $\End(V)$ of all endomorphisms of an \emph{infinite-dimensional} vector space $V$. Let us cite an important result that combines theorems from
\cite{dSPidempotentinfinite1,dSPidempotentinfinite2} (see also \cite{Breaz} and \cite{Shitov}) and
\cite{dSP3involutionsinfinite}:

\begin{theo}
Let $p_1,p_2,p_3,p_4$ be split polynomials of degree $2$ in $\F[t]$.
Let $V$ be an infinite-dimensional vector space. Then:
\begin{enumerate}[(a)]
\item Every element of $\End(V)$ is a $(p_1,p_2,p_3,p_4)$-sum.
\item If $\forall i \in \{1,2,3,4\}, \; p_i(0) \neq 0$, then every element of $\GL(V)$ is a $(p_1,p_2,p_3,p_4)$-product.
\end{enumerate}
\end{theo}

It can be shown that this result still holds if one allows at most two of the polynomials $p_1,\dots,p_4$ to be
irreducible over $\F$, but we do not know if the result still holds with no restriction of irreducibility.

Better still, now we have a classification of the $(p_1,p_2,p_3)$-sums (and of the $(p_1,p_2,p_3)$-products if $p_1(0)\,p_2(0)\,p_3(0) \neq 0$)
when $p_1,p_2,p_3$ are all split, among the endomorphisms that are not finite-rank perturbations of a scalar multiple of the identity
(i.e.\ that are not of the form $\lambda \id_V+w$ with $\lambda \in \F$ and $w \in \End(V)$ of finite rank).
For the remaining endomorphisms of $V$, the classification is probably hopeless because it is deeply connected to corresponding open problems
in the finite-dimensional case. However, in very specific instances the classification is complete:
sums of three square-zero elements of $\End(V)$ are now well understood \cite{dSPidempotentinfinite1},
as well as products of three involutions \cite{dSP3involutionsinfinite}.

In the infinite-dimensional setting, this essentially leaves us with the problem of
decomposing an endomorphism into two quadratic summands or factors.

The aim of the present article is to obtain results in the most simple cases: the products of two involutions (as well as the twin case of products of two
unipotent elements of index $2$), and the sums of two square-zero elements.
In the future, we plan to tackle more complicated situations, e.g.\ sums of two idempotents.

\vskip 3mm
It is important first to recall the known results in the finite-dimensional case.
To this end, we need a useful concept:

\begin{Def}
Let $u \in \End(V)$. We say that $u$ has the \textbf{exchange property}
when there is a splitting $V=V_1 \oplus V_2$ in which $u(V_1) \subset V_2$ and $u(V_2) \subset V_1$.
\end{Def}

\begin{Rem}\label{exchangedirectsumRem}
The exchange property is preserved in taking direct sums of endomorphisms. More precisely, let $u \in \End(V)$,
and let $V=\underset{i \in I}{\bigoplus} V_i$ be a splitting in which each $V_i$ is stable under $u$ and the resulting endomorphism
has the exchange property. Then, for all $i \in I$, we choose a splitting $V_i=F_i \oplus G_i$ such that $u(F_i) \subset G_i$ and
$u(G_i) \subset F_i$, and we set $F:=\underset{i \in I}{\bigoplus} F_i$ and $G:=\underset{i \in I}{\bigoplus} G_i$, so that
$V=F \oplus G$, $u(F) \subset G$ and $u(G) \subset F$.
\end{Rem}

\begin{Rem}\label{exchangesimilarRem}
The exchange property depends on $u$ up to similarity. Remember that an endomorphism $v$ of a vector space $W$
is similar to $u$ if and only if there exists an isomorphism $\varphi : V \overset{\simeq}{\rightarrow} W$ such that
$v=\varphi \circ u \circ \varphi^{-1}$. If $u$ has the exchange property, then we split $V=E \oplus F$ where $u(E) \subset F$
and $u(F) \subset E$; then, by setting $E':=\varphi(E)$ and $F':=\varphi(F)$, we find that $W=E' \oplus F'$, $v(E') \subset F'$ and $v(F') \subset E'$,
and we conclude that $v$ has the exchange property.
\end{Rem}

We are now ready to state the characterization of sums of two square-zero endomorphisms in the finite-dimensional case:

\begin{theo}[Wang-Wu \cite{WangWu} and Botha \cite{Bothasquarezero}]\label{sum2finitetheo}
Let $V$ be a finite-dimensional vector space, and $u \in \End(V)$.
Then the following conditions are equivalent:
\begin{enumerate}[(i)]
\item $u$ is the sum of two square-zero elements of $\End(V)$.
\item The invariant factors of $u$ are all even or odd (as polynomials).
\item The endomorphism $u$ has the exchange property.
\end{enumerate}
Moreover, if $\car(\F) \neq 2$, these conditions are equivalent to:
\begin{enumerate}[(iv)]
\item $u$ is similar to $-u$.
\end{enumerate}
\end{theo}

The next result requires that we define the (rectified) reciprocal polynomial of a monic polynomial $p$:

\begin{Not}
Let $p \in \F[t]$ be a monic polynomial such that $p(0) \neq 0$. Denote by $d$ the degree of $p$. Then we set
$$p^\sharp:=p(0)^{-1}\,t^d\,p(t^{-1}),$$
which we call the \textbf{reciprocal polynomial} of $p$.
\end{Not}

\begin{theo}[Djokovi\'c \cite{Djokovic}, Hoffman and Paige \cite{HoffmanPaige} and Wonenburger \cite{Wonenburger}]\label{prod2finitetheo}
Let $V$ be a finite-dimensional vector space, and $u \in \GL(V)$.
The following conditions are then equivalent:
\begin{enumerate}[(i)]
\item $u$ is the product of two involutions.
\item $u$ is similar to $u^{-1}$.
\item The invariant factors of $u$ are all \textbf{quasi-palindromials}, i.e.\ each one of them $p$ satisfies
$p^\sharp=p$.
\end{enumerate}
\end{theo}

In the infinite-dimensional case, there is no invariant factor to speak of, so the only sensible conjectures seem to be the following ones:

\begin{conj}\label{conj2sum}
Let $V$ be an infinite-dimensional vector space, and $u \in \End(V)$.
Then, $u$ is the sum of two square-zero endomorphisms of $V$ if and only if
$u$ has the exchange property.

Moreover, if $\car(\F) \neq 2$ then this condition is equivalent to having $u$ similar to its opposite.
\end{conj}

\begin{conj}\label{conj2prod}
Let $V$ be an infinite-dimensional vector space, and $u \in \GL(V)$.
Then $u$ is the product of two involutions in $\GL(V)$ if and only if $u$ is similar to its inverse.
\end{conj}

Note that in any group the product of two involutions is conjugated to its inverse, but the converse is not true in general.
Moreover, if $u \in \End(V)$ has the exchange property then it is the sum of two square-zero endomorphisms
and it is similar to its opposite. Assume indeed that $V=V_1 \oplus V_2$ and $u(V_1) \subset V_2$
and $u(V_2) \subset V_1$. Denote by $p_1$ (respectively, by $p_2$) the projection of $V$ on $V_1$ along $V_2$
(respectively, on $V_2$ along $V_1$) then $a:= u \, p_1$ and $b:=u \, p_2$ have square zero and
$u=a+b$. Moreover, one checks that $-u=(p_1-p_2) u (p_1-p_2)$ and $(p_1-p_2)^2=\id_V$.
Hence, one of the main difficult parts in Conjecture \ref{conj2sum} consists in proving that the sum of two square-zero endomorphisms has the exchange property.

\subsection{Locally finite endomorphisms}

In this article, we will validate the above conjectures under rather stringent assumptions on $V$ and $u$:
we will assume that $V$ has countable dimension and that $u$ is locally finite:

\begin{Def}
Let $u$ be an endomorphism of a vector space $V$ over $\F$.
This endows $V$ with a structure of $\F[t]$-module (so that $t.x=u(x)$ for all $x \in V$), and we denote by $V^u$ the resulting module.
We say that $u$ is \textbf{locally finite} whenever $V^u$ is a torsion module or,
in other words, $(u^k(x))_{k \in \N}$ is linearly dependent over $\F$ for all $x \in V$.
We say that $u$ is \textbf{locally nilpotent} whenever $\forall x \in V, \; \exists n \in \N : \; u^n(x)=0$.
\end{Def}

At this point, the reader might wonder why we choose to restrict the study to locally finite endomorphisms of countable-dimensional spaces.
The reason is simple: in the finite-dimensional case, all the known proofs of Theorems \ref{sum2finitetheo} and \ref{prod2finitetheo}
involve a deep understanding of the conjugacy classes in the algebra $\End(V)$. For example, they require at least the Jordan normal form of a nilpotent endomorphism. In contrast, in the infinite-dimensional case no clear understanding of the conjugacy classes is known beyond a very limited scope.
The ``locally finite+countable dimension" case is essentially the only one where there exists a sensible way of classifying the endomorphisms up to similarity.

Now we can state our main results:

\begin{theo}\label{2squarezeroinfinitetheo}
Let $V$ be a vector space with countable dimension, and $u$ be a locally finite endomorphism of it.
The following conditions are equivalent:
\begin{enumerate}[(i)]
\item $u$ is the sum of two square-zero endomorphisms of $V$.
\item $u$ has the exchange property.
\end{enumerate}
If in addition $\car(\F) \neq 2$ then these conditions are equivalent to:
\begin{itemize}
\item[(iii)] $u$ is similar to $-u$.
\end{itemize}
\end{theo}

\begin{theo}\label{2involinfinitetheo}
Let $V$ be a vector space with countable dimension, and $u$ be a locally finite automorphism of it.
Then $u$ is the product of two involutions in $\GL(V)$ if and only if $u$ is similar to its inverse.
\end{theo}

We will also characterize the above properties in terms of the so-called \textbf{Kaplansky invariants} of $u$ (see Section
\ref{Kaplanskysection} for a review of those invariants), which will lead to the following interesting special cases:

\begin{theo}\label{nilpotentsum2theo}
Let $V$ be a vector space with countable dimension.
Then every locally nilpotent endomorphism of $V$ is the sum of two square-zero endomorphisms of $V$.
\end{theo}

\begin{theo}\label{unipotentprod2theo}
Let $V$ be a vector space with countable dimension, and $u$ be a locally nilpotent endomorphism of $V$.
Then $\id_V+u$ is the product of two involutions of $V$, and it is also the product of two unipotent endomorphisms of index $2$.
\end{theo}

These two results have interesting interpretations in terms of infinite matrices.
Let $V$ be a vector space with \emph{infinite} countable dimension.
A column-finite matrix indexed over $\N$ and with entries in $\F$ is a family $(a_{i,j})_{(i,j)\in \N^2}$
of elements of $\F$ such that $\{i \in \N : \; a_{i,j} \neq 0\}$ is finite for all $j \in \N$.
This defines a vector space $\calM_{Cf}(\F)$, and it can be endowed with a structure of $\F$-algebra
by defining the matrix product just like in the finite case. We denote by $\GL_{Cf}(\F)$ its group of invertible elements.

Let $\bfB=(e_n)_{n \in \N}$ be a basis of $V$ indexed over $\N$.
To every endomorphism $u$ of $V$, we assign the matrix $M_\bfB(u):=(u_{i,j})_{(i,j)\in \N^2}$ in which
$u(e_j)=\sum_{i \in \N} u_{i,j}\, e_i$ for all $j \in \N$. This is a column-finite matrix,
and the mapping $u \in \End(V) \mapsto (u_{i,j})_{(i,j)\in \N^2}\in \calM_{Cf}(\F)$
is easily shown to be an isomorphism of $\F$-algebras.

Then one can prove the following results:
\begin{itemize}
\item An endomorphism $u$ of $V$ is locally nilpotent if and only if there exists a basis $\bfB=(e_n)_{n \in \N}$ of $V$
such that $M_\bfB(u)=(u_{i,j})$ is strictly upper-triangular, i.e.\ $u_{i,j}=0$ for all $i \geq j$.
\item An endomorphism $u$ of $V$ is locally finite if and only if there exists a basis $\bfB=(e_n)_{n \in \N}$ of $V$
such that $M_\bfB(u)=(u_{i,j})$ is block upper-triangular, i.e.\ there exists an increasing sequence $(n_i)_{i \in \N}$
of non-negative integers such that, for all $i \in \N$, for all $k > n_{i+1}$ and all $l \leq n_{i}$, one has $u_{k,l}=0$.
\end{itemize}

In particular, Theorems \ref{nilpotentsum2theo} and \ref{unipotentprod2theo}
yield the following results, where we denote by $I_\N$ the unity of the ring $\calM_{Cf}(\F)$:

\begin{theo}\label{strictlyupper2sumtheo}
Let $A$ be a strictly upper-triangular matrix of $\calM_{Cf}(\F)$.
Then $A$ is the sum of two square-zero elements of $\calM_{Cf}(\F)$.
\end{theo}

\begin{theo}\label{upper2prodtheo}
Let $A$ be a strictly upper-triangular matrix of $\calM_{Cf}(\F)$.
Then $I_\N+A$ is the product of two involutions in $\GL_{Cf}(\F)$.
\end{theo}

We will also obtain the following result:

\begin{theo}\label{upper2prodbistheo}
Let $U$ be an upper-triangular matrix of $\calM_{Cf}(\F)$ in which every diagonal entry equals $1$ or $-1$.
Then $U$ is the product of two involutions in $\GL_{Cf}(\F)$.
\end{theo}

In \cite{Slowik}, Theorem \ref{strictlyupper2sumtheo} was stated and given an incorrect proof:
this claim was later retracted \cite{Slowikcorr}. In the appendix of the present article,
we shall explain why the approach taken in \cite{Slowik} is hopeless.

\vskip 3mm
The remainder of the article is structured as follows:
in Section \ref{Kaplanskysection}, we recall all the results on Kaplansky invariants that are necessary for our proofs.
The next two sections are devoted to technical results that are used in the proofs of Theorems \ref{2squarezeroinfinitetheo}
and \ref{2involinfinitetheo}: in Section \ref{quadraticreviewSection}, we review some basic results on the sum and the product of two quadratic elements
of an algebra; in Section \ref{blockoperatorsSection}, we consider $2$ by $2$ block operators and for specific types of such operators we prove various decomposability results (into the sum of two square-zero endomorphisms, or into the product of two involutions or of two unipotent operators of index $2$). Theorems \ref{2squarezeroinfinitetheo} and \ref{2involinfinitetheo} are proved in Sections \ref{2squarezeroSection} and \ref{2involutionSection}, respectively: there, for each decomposition problem we also give a characterization solely in terms of Kaplansky invariants.
Finally, in Section \ref{2unipotentSection}, we characterize the products of two unipotent endomorphisms of index $2$ among the
locally finite endomorphisms of a countable-dimensional vector space.

In the appendix, we shall dissect the flaws of the strategy undertaken in \cite{Slowik} to prove Theorem \ref{nilpotentsum2theo}.

\section{A review of Kaplansky invariants of locally-finite endomorphisms of countable dimensional spaces}\label{Kaplanskysection}

\subsection{The decomposition into $p$-primary parts}

Let $u$ be an endomorphism of a vector space $V$. Remember that $V^u$ denotes the
$\F[t]$-module $(V,+,\bullet)$ where the scalar multiplication $\bullet$ is defined by
$p \bullet x:=p(u)[x]$ for all $p \in \F[t]$ and all $x \in V$.

In general, the torsion of $u$ is the submodule
$$\Tor(u):=\bigl\{x \in V : \; \exists r \in \F[t] \setminus \{0\} : r(u)[x]=0\bigr\}$$
of $V^u$. Hence, $u$ is locally finite if and only if $V=\Tor(u)$.
Given a non-zero polynomial $p \in \F[t]$, the \textbf{$p$-torsion of $u$} is defined as
$$\Tor_p(u):=\Ker p(u).$$
It is a submodule of $V^u$.

We denote by $\Irr(\F)$ the set of all monic irreducible polynomials of $\F[t]$.
Given $p \in \Irr(\F)$, the \textbf{$p^\infty$-torsion} of $u$ is defined as
$$\Tor_{p^\infty}(u):=\bigl\{x \in V : \; \exists n \in \N : \; p^n(u)[x]=0\bigr\}=\underset{n \in \N}{\bigcup} \Tor_{p^n}(u).$$
Again, it is a submodule of $V^u$.
Using B\'ezout's theorem, we obtain
$$\Tor(u)=\underset{p \in \Irr(\F)}{\bigoplus} \Tor_{p^\infty}(u).$$
Hence, $u$ is locally finite if and only if
$$V=\underset{p \in \Irr(\F)}{\bigoplus} \Tor_{p^\infty}(u).$$
Given $p \in \Irr(\F)$, we will say that $u$ is \textbf{$p$-primary} when
$\Tor_{p^\infty}(u)=V$ i.e.\ when $p(u)$ is locally nilpotent.

In any case, $u$ stabilizes $\Tor_{p^\infty}(u)$
and induces a $p$-primary endomorphism of it, called the \textbf{$p$-primary part} of $u$ and denoted by $u_p$.

\subsection{Kaplansky invariants}

From now on, we fix a set $\infty$ which is not an ordinal (say $\infty=\{\{\emptyset\}\}$).
We also denote by $\omega_1$ the first uncountable ordinal.

Let $u \in \End(V)$.
By transfinite induction, we define a non-increasing sequence $p^\alpha(V^u)$ of submodules of $V^u$ as follows:
\begin{itemize}
\item $V_0:=V$;
\item For every ordinal $\alpha$ that has a predecessor, $p^\alpha(V^u):=p(u)\bigl(p^{\alpha-1}(V^u)\bigr)$;
\item For every limit ordinal $\alpha$, $p^\alpha(V^u):=\underset{\beta<\alpha}{\bigcap} p^\beta(V^u)$.
\end{itemize}
The sequence $\bigl(p^\alpha(V^u)\bigr)_\alpha$ terminates and we denote by $p^\infty(V^u)$ its final value.

For every ordinal $\alpha$, the quotient $p^\alpha(V^u)/p^{\alpha+1}(V^u)$ has a natural structure of vector space over the residue
field $\F[t]/(p)$, and for this vector space structure the mapping $p(u)$ induces a linear map
$$p^\alpha(V^u)/p^{\alpha+1}(V^u) \longrightarrow p^{\alpha+1}(V^u)/p^{\alpha+2}(V^u).$$
We denote by $\kappa_{\alpha}(u,p)$ the dimension over $\F[t]/(p)$ of the kernel of that map: it is called the
\textbf{Kaplansky invariant} of order $\alpha$ of $u$ with respect to $p$.
It can also be interpreted as the dimension of the quotient space
$\bigl(\Ker p(u) \cap p^\alpha(V^u)\bigr)/\bigl(\Ker p(u) \cap p^{\alpha+1}(V^u)\bigr)$ over $\F[t]/(p)$.

Finally, the intersection $\Ker p(u) \cap p^\infty(V^u)$ is a vector space over $\F[t]/(p)$
and we denote its dimension by $\kappa_\infty(u,p)$.

By induction, one shows that $\Ker p(u) \cap p^\alpha(V^u)=\Ker \bigl(p(u_p)\bigr) \cap p^\alpha\bigl(\Tor_{p^\infty}(u)^{u_p}\bigr)$ for every ordinal $\alpha$, and we deduce that $\kappa_\infty(u_p,p)=\kappa_\infty(u,p)$ for every ordinal $\alpha$, as well as for $\alpha=\infty$.
Thus, $u$ and its $p$-primary part $u_p$ have the same Kaplansky invariants with respect to $p$.

Note finally that if $V$ has countable dimension, then the sequence $\bigl(p^\alpha(V^u)\bigr)_\alpha$ terminates at some countable ordinal,
and in particular $\kappa_\alpha(u,p)=0$ for every uncountable ordinal $\alpha$. In particular, in the countable-dimensional setting
we shall only consider the invariants $\kappa_\alpha(u,p)$ where $\alpha$ belongs to $\omega_1 \cup \{\infty \}$.

\subsection{The classification theorem}

\begin{theo}[Ulm-Mackey-Kaplansky, see \cite{Fuchs,Kaplansky}]\label{KaplanskyTheorem}
Let $u,v$ be locally finite endomorphisms of vector spaces with countable dimension.
Then $u$ and $v$ are similar if and only if
$$\forall p \in \Irr(\F), \; \forall \alpha \in \omega_1 \cup \{\infty\}, \;
\kappa_\alpha(u,p)=\kappa_\alpha(v,p).$$
\end{theo}

\subsection{Adapted bases}\label{adaptedbasesSection}

The following results are generally attributed to Zippin, who proved analogous statements in the theory
of countably-generated torsion modules over $\Z$ (i.e.\ countably-generated abelian groups).

\begin{theo}[Adapted basis theorem]\label{adaptedbasisTheo}
Let $p \in \Irr(\F)$ have degree $d$, and let $u$ be a $p$-primary endomorphism of a countable-dimensional vector space $V$.
Then there is a family $(e_i)_{i \in I}$ of vectors of $V$ that is \textbf{adapted to $u$} in the following sense:
\begin{itemize}
\item The family $(u^k(e_i))_{i \in I, \; 0 \leq k<d}$ is a basis of the vector space $V$;
\item For all $i \in I$, either $p(u).e_i=0$ or $p(u).e_i=e_j$ for some $j \in I$.
\end{itemize}
\end{theo}

We will need the following special case:

\begin{theo}[Adapted basis theorem for locally nilpotent endomorphisms]\label{nilpotentadaptedbasistheo}
Let $u$ be a locally nilpotent endomorphism of a countable-dimensional vector space $V$.
Then there is a basis $\bfB$ of $V$ such that $u$ maps every vector of $\bfB$
to either a vector of $\bfB$ or the zero vector.
\end{theo}

Finally, we need to identify the possible Kaplansky invariants for a $p$-primary endomorphism:

\begin{Def}
Let $(m_\alpha)_{\alpha \in \omega_1 \cup \{\infty\}}$ be a family of countable cardinals.
We say that it is \textbf{admissible} whenever it satisfies the following two conditions:
\begin{enumerate}[(i)]
\item The set $\{\alpha \in \omega_1 : m_\alpha \neq 0\}$ is countable.
\item For every ordinal $\alpha \in \omega_1$, if $m_{\alpha+k}=0$ for all $k \in \N$ then
$m_\beta=0$ for every ordinal $\beta \in \omega_1$ such that $\alpha \leq \beta$.
\end{enumerate}
\end{Def}

\begin{theo}\label{admissiblesequenceTheo}
Let $(m_\alpha)_{\alpha \in \omega_1 \cup \{\infty\}}$ be a family of countable cardinals, and $p \in \Irr(\F)$.
The following conditions are equivalent:
\begin{enumerate}[(i)]
\item There exists a vector space $V$ with countable dimension, together with a $p$-primary endomorphism $u$ of $V$
such that $\kappa_\alpha(u,p)=m_\alpha$ for all $\alpha \in \omega_1 \cup \{\infty\}$.
\item The family $(m_\alpha)_{\alpha \in \omega_1 \cup \{\infty\}}$ is admissible.
\end{enumerate}
\end{theo}

\subsection{Some consequences on the similarity of locally finite endomorphisms}

\begin{Not}
Let $\lambda \in \F \setminus \{0\}$, and let $p \in \F[t]$ be monic with degree $d$.
Then we set
$$H_\lambda(p):=\lambda^{-d}\,p(\lambda t).$$
\end{Not}

\begin{prop}\label{invariantsH}
Let $u \in \End(V)$ and let $p \in \Irr(\F)$ and $\lambda \in \F \setminus \{0\}$.
Set $q:=H_\lambda(p)$.
Then $\kappa_\alpha(\lambda u,p)=\kappa_\alpha (u,q)$ for every $\alpha \in \omega_1 \cup \{\infty\}$.
\end{prop}

\begin{proof}
Set $v:=\lambda u$.
We note that $p(v)=\lambda^d q(u)$ and $\lambda^d \neq 0$.
Thus, by transfinite induction $p^\alpha(V^v)=q^\alpha(V^u)$ for every ordinal $\alpha$, and by taking the intersection we deduce that
$p^\infty(V^v)=q^\infty(V^u)$.
Moreover, $\Ker p(v)=\Ker q(u)$.

Let $\alpha$ be an ordinal.
The mapping
$$p^\alpha(V^v)/p^{\alpha+1}(V^v) \rightarrow p^{\alpha+1}(V^v)/p^{\alpha+2}(V^v)$$
induced by $p(v)$ is a non-zero scalar multiple of the mapping
$$q^\alpha(V^u)/q^{\alpha+1}(V^u) \rightarrow q^{\alpha+1}(V^u)/q^{\alpha+2}(V^u)$$
induced by $q(u)$. Hence, they have the same kernel, denoted by $W_\alpha$.
Since $p$ and $q$ have the same degree, which we denote by $d$, the dimension of $W_\alpha$ over $\F$ equals both $d\, \kappa_\alpha(v,p)$ and
$d\, \kappa_\alpha(u,q)$, and hence $\kappa_\alpha(v,p)=\kappa_\alpha(u,q)$ because $d$ is a non-zero integer.

The case $\alpha=\infty$ is dealt with likewise.
\end{proof}

\begin{prop}
Let $u \in \GL(V)$ and let $p \in \Irr(\F) \setminus \{t\}$.
Then $\kappa_\alpha(u^{-1},p)=\kappa_\alpha (u,p^\sharp)$ for every $\alpha \in \omega_1 \cup \{\infty\}$.
\end{prop}

\begin{proof}
Set $v:=u^{-1}$ and $q:=p^\sharp$.
Here $p(v)=p(0)\,q(u)\,u^{-d}=p(0)\,u^{-d}\,q(u)$. By transfinite induction, one shows that for each ordinal
$\alpha$, the subspace $p^\alpha (V^u)$ is stable under $u$ and that $u$ induces an automorphism of it.
Then, by another transfinite induction, one deduces that $p^\alpha (V^v)=q^\alpha(V^u)$ for every ordinal $\alpha$.
Let $\alpha$ be an ordinal.
The mapping
$$p^\alpha(V^v)/p^{\alpha+1}(V^v) \rightarrow p^{\alpha+1}(V^v)/p^{\alpha+2}(V^v)$$
induced by $p(v)$ is the composite of the mapping
$$q^\alpha(V^u)/q^{\alpha+1}(V^u) \rightarrow q^{\alpha+1}(V^u)/q^{\alpha+2}(V^u)$$
induced by $q(u)$ and of the automorphism of $q^{\alpha+1}(V^u)/q^{\alpha+2}(V^u)$ induced by
$p(0) u^{-d}$, and hence they have the same kernel.
Like in the proof of Proposition \ref{invariantsH}, we use the observation that $\deg(q)=\deg(p)$ to conclude that
$\kappa_\alpha(u^{-1},p)=\kappa_\alpha (u,q)$.
The case $\alpha=\infty$ is dealt with similarly.
\end{proof}

Using Kaplansky's theorem, we deduce the following two corollaries:

\begin{cor}\label{simtooppositeCor}
Let $V$ be a vector space with countable dimension, and $u \in \End(V)$ be locally finite.
The following conditions are equivalent:
\begin{enumerate}[(i)]
\item $u$ is similar to $-u$.
\item For every polynomial $p \in \Irr(\F)$,
one has $\forall \alpha \in \omega_1 \cup \{\infty\}, \; \kappa_\alpha\bigl(u,H_{-1}(p)\bigr)=\kappa_\alpha(u,p)$.
\end{enumerate}
\end{cor}

\begin{cor}\label{simtoinverseCor}
Let $V$ be a vector space with countable dimension, and $u \in \GL(V)$ be locally finite.
The following conditions are equivalent:
\begin{enumerate}[(i)]
\item $u$ is similar to $u^{-1}$.
\item For every polynomial $p \in \Irr(\F) \setminus \{t\}$,
one has $\forall \alpha \in \omega_1 \cup \{\infty\}, \; \kappa_\alpha(u,p^\sharp)=\kappa_\alpha(u,p)$.
\end{enumerate}
\end{cor}

\subsection{Additional technical results}

Let $p \in \Irr(\F)$, and let $u \in \End(V)$ be $p$-primary.
Let $\alpha$ be an ordinal. Remember that $\kappa_\alpha(u,p)$ is defined as the dimension over $\K:=\F[t]/(p)$
of the kernel of the $\K$-linear map
$$p^{\alpha}(V^u)/p^{\alpha+1}(V^u) \longrightarrow p^{\alpha+1}(V^u)/p^{\alpha+2}(V^u)$$
induced by $p(u)$.

\begin{lemma}\label{sumdimLemma}
Let $p \in \Irr(\F)$, and let $u \in \End(V)$ be $p$-primary. Let $k$ be a positive integer.
Let $\alpha$ be an ordinal. Denote by $d$ the degree of $p$.
Then the kernel of the $\F$-linear map
$$p^{\alpha}(V^u)/p^{\alpha+k}(V^u) \longrightarrow p^{\alpha+1}(V^u)/p^{\alpha+k+1}(V^u)$$
induced by $p(u)$ has its dimension over $\F$ equal to
$$d \sum_{j=0}^{k-1} \kappa_{\alpha+j}(u,p).$$
\end{lemma}

To prove this result, we will use the following basic result:

\begin{lemma}[Exact sequence lemma]\label{ExactSequenceLemma}
In the category of finite-dimensional vector spaces over $\F$, let
$$\xymatrix{
\{0\} \ar[r] & A_1 \ar@{^{(}->}[r]^{u_1} \ar@{->>}[d]_a & B_1 \ar@{->>}[r]^{v_1} \ar[d]^b & C_1 \ar[r] \ar[d]^{c} & \{0\} \\
\{0\} \ar[r] & A_2 \ar@{^{(}->}[r]_{u_2} & B_2 \ar@{->>}[r]_{v_2}  & C_2 \ar[r] & \{0\}
}$$
be a commutative diagram in which the horizontal sequences are exact
and the vertical morphism $a$ is surjective. Then
$$\dim \Ker b=\dim \Ker a+\dim \Ker c.$$
\end{lemma}

\begin{proof}
This is a special case of the snake lemma (Lemma 9.1, chapter III of \cite{Lang}).
By this lemma, there is an exact sequence
$$\{0\} \rightarrow \Ker a \overset{u_0}{\rightarrow} \Ker b \overset{v_0}{\rightarrow} \Ker c \rightarrow
A_2/\im a \rightarrow B_2/\im b \rightarrow C_2/\im c \rightarrow \{0\}$$
in which $u_0$ and $v_0$ are induced by $u_1$ and $v_1$, respectively.
Since $a$ is surjective, this yields an exact sequence
$$\xymatrix{
\{0\} \ar[r] & \Ker a \ar@{^{(}->}[r]^{u_0}  & \Ker b \ar@{->>}[r]^{v_0} & \Ker c \ar[r] & \{0\}
}.$$
From this observation, the result follows immediately.
\end{proof}

\begin{proof}[Proof of Lemma \ref{sumdimLemma}]
We prove the result by induction on $k$. The case $k=1$ is obvious.
Assume that $k>1$. Then we apply Lemma \ref{ExactSequenceLemma} to the commutative diagram
$$\xymatrix{
p^{\alpha+1}(V^u)/p^{\alpha+k+1}(V^u) \ar@{^{(}->}[r] \ar[d]^a &
p^{\alpha}(V^u)/p^{\alpha+k+1}(V^u)  \ar@{->>}[r] \ar[d]^b &
p^{\alpha}(V^u)/p^{\alpha+1}(V^u)  \ar[d]^c  \\
p^{\alpha+2}(V^u)/p^{\alpha+k+2}(V^u)
 \ar@{^{(}->}[r] &
 p^{\alpha+1}(V^u)/p^{\alpha+k+2}(V^u)
  \ar@{->>}[r]  &
  p^{\alpha+1}(V^u)/p^{\alpha+2}(V^u)
   }$$
in which the horizontal morphisms are the canonical ones, and the vertical morphisms are all induced by $p(u)$.
However $\dim_\F \Ker c=d\,\kappa_\alpha(u,p)$, and by induction $\dim_\F \Ker a=d \underset{j=0}{\overset{k-1}{\sum}} \kappa_{(\alpha+1)+j}(u,p)$.
Hence,
$$\dim_\F \Ker b=d \underset{j=0}{\overset{k}{\sum}} \kappa_{\alpha+j}(u,p).$$
\end{proof}

Remember that the dimension of the kernel of the composite of two surjective linear maps equals the sum of the dimensions
of the kernels of the said maps. By induction, this leads to the following result:

\begin{lemma}\label{supersumdimLemma}
Let $p \in \Irr(\F)$, and let $u \in \End(V)$ be $p$-primary. Let $k$ and $l$ be positive integers.
Let $\alpha$ be an ordinal. Denote by $d$ the degree of $p$.
Then the kernel of the $\F$-linear map
$$p^{\alpha}(V^u)/p^{\alpha+k}(V^u) \longrightarrow p^{\alpha+l}(V^u)/p^{\alpha+k+l}(V^u)$$
induced by $p(u)^l$ has its dimension equal to
$$d \sum_{i=0}^{l-1} \sum_{j=0}^{k-1} \kappa_{\alpha+i+j}(u,p).$$
\end{lemma}

\section{A review of basic results on quadratic elements}\label{quadraticreviewSection}

Let $p \in \F[t]$ be monic with degree $2$. The \textbf{trace} of $p$, denoted by $\tr(p)$,
is defined as the opposite of the coefficient of $p$ on $t$
(it is the sum of the roots of $p$, counted with multiplicities). Let $\calA$ be an $\F$-algebra, and
let $a \in \calA$ satisfy $p(a)=0$. Then
$$a\,\bigl(\tr(p)\, 1_\calA-a\bigr)=\bigl(\tr(p)\, 1_\calA-a\bigr)\,a=p(0)\,1_\calA,$$
and we set
$$a^\star:=\tr(p)\, 1_\calA-a,$$
called the \textbf{$p$-conjugate of $a$.} In particular, if $p(0) \neq 0$ then $a$ is invertible and
$a^\star=p(0)\,a^{-1}$.

\begin{lemma}[Basic commutation lemma, lemma 1.1 of \cite{dSPsumprod2}]\label{basicCommutationlemma}
Let $\calA$ be an $\F$-algebra, and let $p,q$ be monic polynomials with degree $2$ in $\F[t]$.
Let $(a,b) \in \calA^2$ satisfy $p(a)=q(b)=0$,
and denote by $a^\star$ and $b^\star$ the $p$-conjugate of $a$ and the $q$-conjugate of $b$, respectively.
Then both $a$ and $b$ commute with $ab^\star+b a^\star$.
\end{lemma}

\begin{cor}\label{commutationCor1}
Let $a$ and $b$ be square-zero elements of an $\F$-algebra. Set $u:=a+b$. Then $a$ and $b$ commute with $u^2$.
\end{cor}

\begin{proof}
Indeed, here we take $p=q=t^2$ and we have $a^\star=-a$ and $b^\star=-b$, so $a$ and $b$ commute with $ab+ba=u^2-a^2-b^2=u^2$.
\end{proof}

\begin{cor}\label{commutationCor2}
Let $a$ and $b$ be elements of an $\F$-algebra $\calA$ such that $(a-1_\calA)^2=0=(b-1_\calA)^2$. Set $u:=ab$.
Then both $a$ and $b$ commute with $u+u^{-1}$.
\end{cor}

\begin{proof}
Here, we take $p=q=(t-1)^2=t^2-2t+1$ and we note that $(b^{-1}-1_\calA)^2=0$, and
that $b=(b^{-1})^{-1}=(b^{-1})^\star$ and $a^{-1}=a^\star$. Hence,
$a$ and $b^{-1}$ commute with $ab+b^{-1}a^{-1}=u+u^{-1}$, which yields the claimed result.
\end{proof}

\section{Block operators}\label{blockoperatorsSection}

Let $V$ be a vector space. We define the canonical injections
$$i_1 : x \in V \mapsto (x,0) \in V^2 \quad \text{and} \quad  i_2 : x \in V \mapsto (0,x) \in V^2$$
and the canonical projections
$$\pi_1 : (x,y) \in V^2 \mapsto x \in V \quad \text{and} \quad  \pi_2 : (x,y) \in V^2 \mapsto y \in V.$$
The mapping
$$M : \begin{cases}
\End(V^2) & \longrightarrow \Mat_2(\End(V)) \\
u & \longmapsto \begin{bmatrix}
\pi_1\circ  u\circ  i_1 & \pi_1\circ u \circ i_2 \\
\pi_2\circ u \circ i_1 & \pi_2\circ u \circ i_2
\end{bmatrix}
\end{cases}$$
is easily shown to be an isomorphism of $\F$-algebras.

Let $a \in \End(V)$. We denote by $K(a)$ the endomorphism of $V^2$ such that
$$M(K(a))=\begin{bmatrix}
0 & a \\
\id_V & 0
\end{bmatrix},$$
and by $L(a)$ the endomorphism of $V^2$ such that
$$M(L(a))=\begin{bmatrix}
0 & -\id_V \\
\id_V & a
\end{bmatrix}.$$

\begin{lemma}\label{exchangediagonalLemma}
Assume that $\car(\F) \neq 2$.
Let $a \in \End(V)$, and set $b \in \End(V^2)$ such that $M(b)=\begin{bmatrix}
a & 0 \\
0 & -a
\end{bmatrix}$.
Then $b$ has the exchange property.
\end{lemma}

\begin{proof}
Set $W_1:=\bigl\{(x,x) \mid x \in V\bigr\}$ and $W_2:=\bigl\{(x,-x) \mid x \in V\bigr\}$.
As $\car(\F) \neq 2$, one sees that $V^2=W_1 \oplus W_2$. Noting that $b : (x,y) \mapsto \bigl(a(x),-a(y)\bigr)$, one
checks that $b(W_1) \subset W_2$ and $b(W_2) \subset W_1$.
\end{proof}

\begin{lemma}\label{inversediagonalLemma}
Let $a$ be an element of an $\F$-algebra $\calA$.
Then the matrix $B:=\begin{bmatrix}
a & 0 \\
0 & a^{-1}
\end{bmatrix}$ is the product of two involutions in $\Mat_2(\calA)$.
Moreover, if $a^2-1_\calA$ is invertible, then $B$ is also the product of two unipotent elements of index $2$ in
$\Mat_2(\calA)$.
\end{lemma}

\begin{proof}
Firstly, one sees that
$$\begin{bmatrix}
a & 0 \\
0 & a^{-1}
\end{bmatrix}=\begin{bmatrix}
0 & 1_\calA \\
1_\calA & 0
\end{bmatrix} \begin{bmatrix}
0 & a^{-1} \\
a & 0
\end{bmatrix},$$
which proves the first result.

Next, assume that $a^2-1_\calA$ is invertible.
Setting $$T:=\begin{bmatrix}
1_\calA & a-1_\calA \\
1_\calA & a^{-1}-1_\calA
\end{bmatrix},$$
one computes that $T$ is invertible, that
$$T^{-1}=(a^{-1}-a)^{-1} \begin{bmatrix}
a^{-1}-1_\calA & -(a-1_\calA) \\
-1_\calA & 1_\calA
\end{bmatrix}$$
and that
$$T^{-1}BT= \begin{bmatrix}
1_\calA & a+a^{-1}-2.1_\calA \\
1_\calA & a+a^{-1}-1_\calA
\end{bmatrix}.$$
Hence,
$$T^{-1}BT= \begin{bmatrix}
1_\calA & 0 \\
1_\calA & 1_\calA
\end{bmatrix} \,\begin{bmatrix}
1_\calA & a+a^{-1}-2.1_\calA \\
0 & 1_\calA
\end{bmatrix},$$
which is the product of two unipotent elements of index
 $2$. We conclude that so is $B$.
\end{proof}

\section{Sums of two square-zero endomorphisms}\label{2squarezeroSection}

\subsection{Main results}

Our aim here is to characterize the sums of two square-zero endomorphisms among the locally finite endomorphisms
of a vector space with countable dimension. The main goal is to prove Theorem \ref{2squarezeroinfinitetheo},
in which the implications (ii) $\Rightarrow$ (i) and (ii) $\Rightarrow$ (iii) are already known.
Combining this theorem with Corollary \ref{simtooppositeCor} will yield the following characterization over fields
with characteristic different from $2$:

\begin{cor}\label{sum2carnot2cor}
Assume that $\car(\F) \neq 2$. Let $V$ be a vector space with countable dimension, and $u \in \End(V)$ be locally finite.
The following conditions are equivalent:
\begin{enumerate}[(i)]
\item $u$ is the sum of two square-zero endomorphisms of $V$.
\item $u$ has the exchange property.
\item $u$ is similar to its opposite.
\item For every polynomial $p \in \Irr(\F)$,
one has $\forall \alpha \in \omega_1 \cup \{\infty\}, \; \kappa_\alpha\bigl(u,H_{-1}(p)\bigr)=\kappa_\alpha(u,p)$.
\end{enumerate}
\end{cor}

Over fields with characteristic $2$, conditions (iii) and (iv) in this corollary are always satisfied (and trivially so!),
and hence it is also interesting to characterize sums of two square-zero endomorphisms in terms of Kaplansky invariants.

Remember that every ordinal $\alpha$ has a unique expression as the sum $\delta+k$
where $\delta$ is an ordinal with no predecessor and $k \in \N$. We say that $\alpha$ is \textbf{even}
when $k$ is even, and \textbf{odd} otherwise. Here is our result for fields with characteristic $2$:

\begin{theo}\label{sum2car2theo}
Let $V$ be a vector space with countable dimension. Assume that $\car(\F)=2$.
Let $u \in \End(V)$ be locally finite. Then the following conditions are equivalent:
\begin{enumerate}[(i)]
\item $u$ is the sum of two square-zero endomorphisms of $V$.
\item $u$ has the exchange property.
\item For every non-even polynomial $p\in \Irr(\F) \setminus \{t\}$ and for every even ordinal $\alpha$, one has $\kappa_\alpha(u,p)=0$.
\end{enumerate}
\end{theo}

The remainder of this section is organized as follows:
in Section \ref{exchangespecialcasesSection}, we prove that endomorphisms of very specific types
have the exchange property, most notably the locally nilpotent ones, and
we also compute the Kaplansky invariants of block operators of type $K(a)$.
Section \ref{invertible2squarezeroSection} consists of a simple lemma on automorphisms that
are the sum of two square-zero endomorphisms.
The proofs of Theorems \ref{2squarezeroinfinitetheo} and \ref{sum2car2theo}
are given in the next three sections: in Section \ref{squarezeroiImpliesiiSection}, we prove the implication (i) $\Rightarrow$ (ii), while the
implication (iii) $\Rightarrow$ (i) is proved in Section \ref{squarezeroiiiImpliesiSection} for fields with characteristic different from $2$; finally, the proof of Theorem \ref{sum2car2theo} is carried out in Section \ref{sum2car2Section}.

\subsection{From the exchange property to conditions on the Kaplansky invariants}\label{exchangespecialcasesSection}

We start with the case of a locally nilpotent endomorphism. Note that the following result implies Theorem \ref{nilpotentsum2theo}.

\begin{lemma}\label{localnilpexchangelemma}
Let $u$ be a locally nilpotent endomorphism of a vector space $V$ with countable dimension.
Then $u$ has the exchange property.
\end{lemma}

\begin{proof}
We note that $u$ is $t$-primary. By Theorem \ref{nilpotentadaptedbasistheo}, we can take a basis $(x_k)_{k \in K}$ of $V$
such that, for all $k \in K$, the vector $u(x_k)$ equals zero or some $x_l$.
For $y \in V$, denote by $\nu(y)$ the least integer $p \geq 0$ such that $u^p(y)=0$, and note that
$\nu(u(y))=\nu(y)-1$ whenever $y \neq 0$.
Set then
$$V_0:=\Vect\{x_k \mid k \in K \; \text{such that}\; \nu(x_k)\; \text{is even}\bigr\}$$
and
$$V_1:=\Vect\{x_k \mid k \in K \; \text{such that}\; \nu(x_k)\; \text{is odd}\bigr\}.$$
Clearly, $V =V_0 \oplus V_1$ and the above remark shows that $u(V_0) \subset V_1$ and $u(V_1) \subset V_0$.
\end{proof}

Next, we adapt the method to obtain the following result:

\begin{lemma}\label{evenexchangelemma}
Let $p \in \Irr(\F)$ be even, and let $u$ be a $p$-primary endomorphism of a vector space $V$ with countable dimension.
Then $u$ has the exchange property.
\end{lemma}

\begin{proof}
Denote by $d$ the degree of $p$ (this is an even integer).
By Theorem \ref{adaptedbasisTheo}, we can take a basis $(u^i(x_k))_{0 \leq i<d,k \in K}$ of $V$ that is adapted to $u$.
Then, we consider the subspaces
$$V_0:=\Vect\bigl\{u^i(x_k)\mid k \in K, 0 \leq i<d, \; i \; \text{even}\bigr\}$$
and
$$V_1:=\Vect\bigl\{u^i(x_k)\mid k \in K, 0 \leq i<d, \; i \; \text{odd}\bigr\}.$$
Obviously $V=V_0 \oplus V_1$ and $u$ maps $V_0$ into $V_1$.
It remains to prove that $u$ maps $V_1$ into $V_0$, and for this it suffices to see that
$u^d(x_k) \in V_0$ for all $k \in K$. Let $k \in K$. On the one hand $p(u)[x_k] \in V_0$ because
$p(u)[x_k]$ equals either the zero vector or some $x_l$. On the other hand,
as $p$ is even we have
$$u^d(x_k)=p(u)[x_k]+\underset{i=0}{\overset{\frac{d}{2}-1}{\sum}} a_i \, u^{2i}(x_k),$$
for some list $\bigl(a_0,\dots,a_{\frac{d}{2}-1}\bigr) \in \F^{d/2}$.
On the right-hand side of that equality sits a vector of $V_0$. We conclude that $u(V_1) \subset V_0$,
which shows that $u$ has the exchange property.
\end{proof}

Alternatively, the result can be obtained thanks to Kaplansky's theorem, to Theorem \ref{admissiblesequenceTheo}
and to the next lemma, whose proof we leave to the interested reader:

\begin{lemma}\label{Klemma1}
Let $p \in \Irr(\F)$ be even and let $a \in \End(V)$.
Set $b:=K(a) \in \End(V^2)$, and define $q \in \F[t]$ by $p(t)=q(t^2)$.
Then $q$ is irreducible over $\F$. Moreover, if $a$ is $q$-primary, then $b$ is $p$-primary
and, for every $\alpha$, ordinal or $\infty$,
$$\kappa_{\alpha}(b,p)=\kappa_{\alpha}(a,q).$$
\end{lemma}

%

\begin{lemma}\label{Klemma2}
Assume that $\car(\F)=2$. Let $p \in \Irr(\F) \setminus \{t\}$ be non-even and let $a \in \End(V)$.
Set $b:=K(a)$ and define $q \in \F[t]$ by $q(t^2)=p(t)^2$.
Then $q$ is irreducible over $\F$. Moreover, if $a$ is $q$-primary, then $b$ is $p$-primary and, for every ordinal $\delta$ with no predecessor and every $k \in \N$,
$$\kappa_{\delta+2k+1}(b,p)=\kappa_{\delta+k}(a,q) \quad \text{and} \quad
\kappa_{\delta+2k}(b,p)=0.$$
Finally,
$$\kappa_\infty(b,p)=\kappa_\infty(a,q).$$
\end{lemma}

\begin{proof}[Proof of Lemma \ref{Klemma2}]
Note that the definition of $q$ is correct because $\F$ has characteristic $2$. Denote by $d$ the degree of $p$, and note that
it is also the degree of $q$.

If $q$ were reducible, then we could write $q=q_1q_2$ for non-constant monic polynomials $q_1$ and $q_2$, and hence
by the irreducibility of $p$ we would gather that $q_1(t^2)=p=q_2(t^2)$, contradicting the assumption that $p$ is non-even.
Hence, $q$ is irreducible.

Assume from now on that $a$ is $q$-primary. Since $M(b)^2=\begin{bmatrix}
a & 0 \\
0 & a
\end{bmatrix}$, we see that
$M(p^{2k}(b))=\begin{bmatrix}
q^k(a) & 0 \\
0 & q^k(a)
\end{bmatrix}$ for all $k \in \N$.
For every vector $(x_1,x_2) \in V^2$, we can find $k \in \N$ such that $q^k(a)[x_1]=0=q^k(a)[x_2]$, and hence $(x_1,x_2) \in \Ker p^{2k}(b)$.
It follows that $b$ is $p$-primary.

Next, a transfinite induction shows that, for every ordinal $\delta$ with no predecessor and every integer $k \in \N$,
$$p^{\delta+2k}((V^2)^b)=\bigl(q^{\delta+k}(V^a)\bigr)^2.$$
Taking the intersection over all pairs $(\delta,k)$, we deduce that
$$p^\infty\bigl((V^2)^b\bigr)=(q^\infty(V^a)\bigr)^2.$$

Next, we compute $\kappa_\infty(b,p)$. Since $p(b)$ induces a surjective endomorphism of $p^\infty(b)(V^2)$,
it induces a surjective linear map from $\Ker p^2(b) \cap p^\infty(b)(V^2)$ to $\Ker p(b) \cap p^\infty(b)(V^2)$ whose kernel equals
$\Ker p(b) \cap p^\infty(b)(V^2)$, and it follows that
$$\dim_\F\Bigl[\bigl(\Ker p^2(b) \cap p^\infty\bigl((V^2)^n\bigr)\Bigr]=2 \dim_\F \Bigl[\Ker p(b) \cap p^\infty\bigl((V^2)^b\bigr)\Bigr]=2\,d\, \kappa_\infty(b,p).$$
Clearly, the above expression of $p^2(b)$ shows that
$$\Ker p^2(b) \cap  p^\infty(b)(V^2)=\bigl(\Ker q(a) \cap q^\infty(a)(V)\bigr)^2$$
and hence
$$\dim_\F \Bigl[\Ker p^2(b) \cap  p^\infty\bigl((V^2)^b\bigr)\Bigr]=2 \dim_\F \bigl[\Ker q(a) \cap q^\infty(a)(V)\bigr]=2\,d\, \kappa_\infty(a,q).$$
It easily follows that $\kappa_\infty(b,p)=\kappa_\infty(a,q)$ (one needs to discuss whether $\kappa_\infty(b,p)$ is finite or not).

Next, we compute the Kaplansky invariants of $b$ as functions of those of $a$.
First, let $\alpha$ be an ordinal.
By Lemma \ref{supersumdimLemma}, the dimension over $\F$ of the kernel of the linear map
$$p^{\alpha}\bigl((V^2)^b\bigr)/p^{\alpha+2}\bigl((V^2)^b\bigr) \longrightarrow
p^{\alpha+2}\bigl((V^2)^b\bigr)/p^{\alpha+4}\bigl((V^2)^b\bigr)$$
induced by $p^2(b)$ equals
$$d\,\bigl(\kappa_{\alpha}(b,p)+2\, \kappa_{\alpha+1}(b,p)+\kappa_{\alpha+2}(b,p)\bigr).$$

Next, let $\delta$ be an ordinal with no predecessor, and $k$ be an integer.
Then $p^{\delta+2k}\bigl((V^2)^b\bigr)=\bigl(q^{\delta+k}(V^a)\bigr)^2$,
$p^{\delta+2k+2}\bigl((V^2)^b\bigr)=\bigl(q^{\delta+k+1}(V^a)\bigr)^2$ and
$p^{\delta+2k+4}\bigl((V^2)^b\bigr)=\bigl(q^{\delta+k+2}(V^a)\bigr)^2$.
From the above expression of $q^2(b)$, we gather that the surjective linear mapping
$$p^{\delta+2k}\bigl((V^2)^b\bigr)/p^{\delta+2k+2}\bigl((V^2)^b\bigr)
\longrightarrow p^{\delta+2k+2}\bigl((V^2)^b\bigr)/p^{\delta+2k+4}\bigl((V^2)^b\bigr)$$
induced by $p^2(b)$ is equivalent\footnote{Two linear maps $a : U_1 \rightarrow U_2$ and $b : V_1 \rightarrow V_2$ are called equivalent
when there exist isomorphisms $\varphi : U_1 \overset{\simeq}{\rightarrow} V_1$ and $\psi : U_2 \overset{\simeq}{\rightarrow} V_2$ such that
$b=\psi \circ a \circ \varphi^{-1}$.}
to the direct sum of two copies of the surjective linear mapping
$$q^{\delta+k}(V^a)/q^{\delta+k+1}(V^a) \longrightarrow q^{\delta+k+1}(V^a)/q^{\delta+k+2}(V^a)$$
induced by $q(a)$. From this, we deduce the identity
$$\kappa_{\delta+2k}(b,p)+2\, \kappa_{\delta+2k+1}(b,p)+\kappa_{\delta+2k+2}(b,p)=2\,\kappa_{\delta+k}(a,q).$$

Hence, in order to conclude it remains to prove that
$\kappa_{\delta+2k}(b,p)=0$ for every ordinal $\delta$ with no predecessor and every $k\in \N$.
The main difficulty lies in the case when $\delta=k=0$, and the general case will be seen as an easy corollary of it.

To solve the case when $\delta=k=0$, we need to prove that the linear mapping
$$V/p((V^2)^b) \longrightarrow p((V^2)^b)/p^2((V^2)^b)$$
induced by $p(b)$ is injective.
So, let $(x,y) \in V^2$ be such that $p(b)(x,y) \in \im \bigl(p(b)^2\bigr)$. We need to prove that $(x,y) \in \im p(b)$.
To this end, we split $p=r(t^2)+ts(t^2)$ where $r$ and $s$ are polynomials.
Since $p \neq t$ and $p$ is irreducible, we see that $r \neq 0$.

Next,
$$M(p(b))=\begin{bmatrix}
r(a) & a s(a) \\
s(a) & r(a)
\end{bmatrix} \quad \text{and} \quad
M(p(b)^2)=\begin{bmatrix}
r(a)^2+a s(a)^2 & 0 \\
0 & r(a)^2+a s(a)^2
\end{bmatrix}.$$
Hence there exists $x_1 \in V$ such that
$$r(a)[x]+(as(a))[y]=r(a)^2[x_1]+(a s(a)^2)[x_1].$$

Next, we note that $\deg r<\deg p$. Since $r \neq 0$, we deduce that $p$ does not divide $r$.
Since $a$ is $p$-primary, it follows from B\'ezout's theorem that $r(a)$ is an automorphism of $V$.

Setting $x_0:=x_1$, we obtain
$$r(a)\bigl[x-r(a)[x_0]\bigr]=(as(a))\bigl[s(a)[x_0]-y\bigr]$$
and hence
$$x-r(a)[x_0]=(as(a))\bigl[r(a)^{-1}(s(a)[x_0]-y)\bigr].$$
Setting $y_0:=r(a)^{-1}(s(a)[x_0]-y)$, we derive that $x=r(a)[x_0]+(as(a))[y_0]$.
Besides, $r(a)[y_0]=s(a)[x_0]-y$ leads to $y=s(a)[x_0]+r(a)[y_0]$ because $\car(\F)=2$. Hence, $(x,y)=p(b)(x_0,y_0)$, and the claimed result ensues.

Now, we can conclude. Let $\delta$ be an ordinal with no predecessor, and let $k \in \N$.
Set $W:=q^{\delta+k}(V^a)$. We know that $p^{\delta+2k}\bigl((V^2)^b\bigr)=W^2$.
Applying the previous result, where we replace $V$ with $W$ and $a$ with the $p$-primary endomorphism of $W$ it induces, we obtain that the linear mapping
$$p^{\delta+2k}\bigl((V^2)^b\bigr)/p^{\delta+2k+1}\bigl((V^2)^b\bigr) \longrightarrow
p^{\delta+2k+1}\bigl((V^2)^b\bigr)/p^{\delta+2k+2}\bigl((V^2)^b\bigr)$$
induced by $p(b)$ is injective, which leads to $\kappa_{\delta+2k}(b,p)=0$.
\end{proof}

\subsection{Invertible sums of two square-zero operators}\label{invertible2squarezeroSection}

\begin{lemma}\label{sum2invlemma}
Let $V$ be a vector space and let $u \in \GL(V)$.
Assume that $u=a+b$ for square-zero elements $a,b$ of $\End(V)$.
Then $E=\Ker a \oplus \Ker b$, $\im a=\Ker a$ and $\im b=\Ker b$.
\end{lemma}

\begin{proof}
Since $u$ is bijective, we find $\Ker a \cap \Ker b=\{0\}$ and $\im a+\im b=V$.
Note also that $\im a \subset \Ker a$ and $\im b \subset \Ker b$ because $a^2=b^2=0$.
For all $x \in \Ker a$, we split $x=x_1+y_1$ for some $x_1 \in \im a$ and some $y_1 \in \im b$,
and then $x-x_1 \in \Ker a \cap \Ker b$, leading to $x=x_1 \in \Ker a$. Thus, $\im a=\Ker a$, and likewise $\im b=\Ker b$,
leading to $V=\Ker a \oplus \Ker b$.
\end{proof}

\subsection{Proof of (i) $\Rightarrow$ (ii) in Theorem \ref{2squarezeroinfinitetheo}}\label{squarezeroiImpliesiiSection}

Let $u \in \End(V)$ be locally finite. Assume that $u=a+b$ for some endomorphisms $a,b$ of $V$ such that $a^2=b^2=0$.
By Corollary \ref{commutationCor1}, $u^2$ commutes with $a$ and $b$. We wish to prove that $u$ has the exchange property.

Assume first that $u$ is an automorphism.
By Lemma \ref{sum2invlemma}, we have $V=\Ker a \oplus \Ker b$.
Then, we see that $u(\Ker a) \subset \im b \subset \Ker b$, and $u(\Ker b) \subset \im a\subset \Ker a$.
Hence, $u$ has the exchange property.

Let us now come back to the general case.
Let $p \in \Irr(\F) \setminus \{t\}$. Set $q:=H_{-1}(p)$.
Then $q$ is irreducible. Moreover, $pq$ is obviously even, leading to $pq=r(t^2)$ for some $r \in \F[t]$. Then
$(pq)(u)=r(u^2)$, and since $a$ and $b$ commute with $u^2$ we obtain that $a$ and $b$ stabilize
$\underset{n \in \N}{\bigcup} \Tor_{r^n}(u^2)$.
\begin{itemize}
\item If $q \neq p$, then
$$\underset{n \in \N}{\bigcup} \Tor_{r^n}(u^2)=\Tor_{p^\infty}(u)\oplus \Tor_{q^\infty}(u),$$
and hence $\Tor_{p^\infty}(u)\oplus \Tor_{q^\infty}(u)$ is stable under both $a$ and $b$.

\item If $q=p$, then
$$\Tor_{p^\infty}(u)=\underset{n \in \N}{\bigcup} \Tor_{p^{2n}}(u)=\underset{n \in \N}{\bigcup} \Tor_{r^n}(u^2),$$
and hence $\Tor_{p^\infty}(u)$ is stable under both $a$ and $b$.
\end{itemize}

The above arguments show that $V':=\underset{p \in \Irr(\F) \setminus \{t\}}{\bigoplus} \Tor_{p^\infty}(u)$ is stable under $a$ and $b$, and it follows that the automorphism $u'$ of $V'$ induced by $u$ is the sum of two square-zero endomorphisms.
By the case treated in the above, $u'$ has the exchange property.
Finally, the endomorphism of $\Tor_{t^\infty}(u)$ induced by $u$ is locally nilpotent, so by Lemma \ref{localnilpexchangelemma} it has the exchange property. Using Remark \ref{exchangedirectsumRem}, we conclude that $u$ has the exchange property.

\subsection{Proof of (iii) $\Rightarrow$ (ii) in Theorem \ref{2squarezeroinfinitetheo}}\label{squarezeroiiiImpliesiSection}

Assume that $\car(\F) \neq 2$. Let $u \in \End(V)$ be locally finite and similar to its opposite.
We wish to prove that $u$ has the exchange property.
First of all, $u_t$ is locally nilpotent, and hence it has the exchange property.

Now let $p \in \Irr(\F) \setminus \{t\}$.
If $p$ is even, we already know that $u_p$ has the exchange property (see Lemma \ref{evenexchangelemma}).
Assume now that $p$ is non-even. Set $q:=H_{-1}(p)$. Since $\car(\F) \neq 2$ and $p$ is neither even nor odd,
we get $p \neq q$.
On the other hand, since $u$ is similar to its opposite, we obtain that
$-u_q$ is similar to $u_p$. Hence, the endomorphism $u_{p,q}$ of $\Tor_{p^\infty}(u) \oplus \Tor_{q^\infty}(u)$ induced by
$u$ is similar to the endomorphism $b$ of $\Tor_{p^\infty}(u)^2$ such that $M(b)=\begin{bmatrix}
u_p & 0 \\
0 & -u_p
\end{bmatrix}$. By Lemma \ref{exchangediagonalLemma}, $b$ has the exchange property, and hence so does $u_{p,q}$.

Finally, by writing $\Irr(\F)$ as the union of the orbits of the involution
$p \mapsto H_{-1}(p)$, we apply Remark \ref{exchangedirectsumRem} to conclude that $u$ has the exchange property.

The proof of Theorem \ref{2squarezeroinfinitetheo} is now complete
(remember that the implications (ii) $\Rightarrow$ (i) and (ii) $\Rightarrow$ (iii) have been proved in
the end of Section \ref{problemSection}).

\subsection{Proof of Theorem \ref{sum2car2theo}}\label{sum2car2Section}

Throughout this section, we assume that $\car(\F)=2$. Let $u \in \End(V)$ be locally finite.

Assume that $u=a+b$ for some square-zero endomorphisms $a$ and $b$ of $V$.
Let $p \in \Irr(\F) \setminus \{t\}$ be non-even. We wish to prove that $\kappa_\alpha(u,p)=0$ for every even ordinal $\alpha$.

First of all, we reduce the situation to the one where $u$ is $p$-primary.
We write $p(t)^2=q(t^2)$ for some $q \in \Irr(\F)$ (see Lemma \ref{Klemma2}). Since $a$ and $b$ commute with $u^2$, they also commute with
$q(u^2)=p(u)^2$, and hence they stabilize $\Tor_{p^\infty}(u)$.
Thus, $u_p$ is the sum of two square-zero endomorphisms. Hence, in the remainder of the proof we can assume that $u$ is $p$-primary.

Next, $u$ is now an automorphism, and hence by Lemma \ref{sum2invlemma} we have $V=\Ker a \oplus \Ker b$,
$\im b=\Ker b$ and $\im a=\Ker a$. Hence, $u$ induces an isomorphism $u_1$ from $\Ker a$ to $\Ker b$ and an isomorphism $u_2$ from $\Ker b$ to $\Ker a$.
Through the isomorphism $(x,y) \in (\Ker a)^2 \mapsto x+u_1(y) \in V$,
we obtain that $u$ is similar to the endomorphism $c$ of $(\Ker a)^2$ such that
$M(c)=\begin{bmatrix}
0 & u_2 \circ u_1 \\
\id_{\Ker a} & 0
\end{bmatrix}$. In particular, $c$ is $p$-primary, so by Lemma \ref{Klemma2} we find that
$\kappa_\alpha(c,p)=0$ for every even ordinal $\alpha$.
Hence, $\kappa_\alpha(u,p)=\kappa_\alpha(c,p)=0$ for every even ordinal $\alpha$.

\vskip 3mm
Conversely, assume that $\kappa_\alpha(u,p)=0$ for every even ordinal $\alpha$ and every non-even polynomial $p \in \Irr(\F) \setminus \{t\}$. We wish to prove that $u$ has the exchange property.
By Remark \ref{exchangedirectsumRem} and Lemma \ref{localnilpexchangelemma}, it will suffice to prove that $u_p$
has the exchange property for every $p \in \Irr(\F) \setminus \{t\}$.
This is already known for $p$ even, by Lemma \ref{evenexchangelemma}.

Let now $p \in \Irr(\F) \setminus \{t\}$ be non-even.
Hence, $\kappa_\alpha(u_p,p)=\kappa_\alpha(u,p)=0$ for every even ordinal $\alpha$.
Now, set $m_{\delta+k}:=\kappa_{\delta+2k+1}(u,p)$ for every ordinal $\delta$ with no predecessor and every $k \in \N$. Set also $m_\infty:=\kappa_\infty(u,p)$.
Since $(\kappa_\alpha(u,p))_{\alpha \in \omega_1 \cup \{\infty\}}$ is admissible (see Section \ref{adaptedbasesSection}),
it is clear that so is $(m_\beta)_{\beta \in \omega_1 \cup \{\infty\}}$.
By Theorem \ref{admissiblesequenceTheo}, there exists a vector space $W$ with countable dimension together with a $q$-primary endomorphism $v \in \End(W)$ such that $\kappa_\beta(v,q)=m_\beta$ for every $\beta \in \omega_1 \cup \{\infty\}$.
Using Lemma \ref{Klemma2}, we obtain that the endomorphism $v'$ of $W^2$ such that $M(v')=\begin{bmatrix}
0 & v \\
\id_W & 0
\end{bmatrix}$ is $p$-primary and satisfies
$\kappa_\alpha(v',p)=\kappa_\alpha(u_p,p)$ for all $\alpha \in \omega_1 \cup \{\infty\}$.
By Kaplansky's theorem (Theorem \ref{KaplanskyTheorem}) $v'$ is similar to $u_p$. Yet, it is obvious that $v'$ has the exchange property, and hence so does $u_p$, by Remark \ref{exchangesimilarRem}.

We conclude that $u$ has the exchange property.

Thus, the implications (i) $\Rightarrow$ (iii) and (iii) $\Rightarrow$ (ii) from Theorem \ref{sum2car2theo}
are proved. The implication (ii) $\Rightarrow$ (i) was already known, and hence the proof of Theorem \ref{sum2car2theo}
is complete.

\section{Products of two involutions}\label{2involutionSection}

\subsection{Main result}

The aim of the present section is to prove Theorem \ref{2involinfinitetheo}.
In order to do so, we will prove a deeper result, in which products of two involutions are characterized in terms
of Kaplansky invariants:

\begin{theo}\label{prod2involBigTheo}
Let $V$ be a vector space with countable dimension.
Let $u \in \GL(V)$ be locally finite. The following conditions are equivalent :
\begin{enumerate}[(i)]
\item $u$ is the product of two involutions in the group $\GL(V)$.
\item $u$ is similar to its inverse.
\item For every $p \in \Irr(\F) \setminus \{t\}$ and every $\alpha \in \omega_1 \cup \{\infty\}$, one has
$\kappa_\alpha(u,p)=\kappa_\alpha(u,p^\sharp)$.
\end{enumerate}
\end{theo}

Here, the equivalence between conditions (ii) and (iii) is known (see Corollary \ref{simtoinverseCor}), and the implication (i) $\Rightarrow$ (ii) is known. Hence, it suffices to prove that condition (iii) implies condition (i).
To this end, an important remark will be useful:

\begin{Rem}\label{superdirectsumremark}
Let $u \in \End(V)$ and let $V=\underset{i \in I}{\bigoplus} V_i$ be such that each $V_i$ is stable under $u$, and
denote by $u_i$ the resulting endomorphism of $V_i$.
Let $p_1,\dots,p_n$ be polynomials in $\F[t]$. Assume that each $u_i$ is a
$(p_1,\dots,p_n)$-product. Then $u$ is
a $(p_1,\dots,p_n)$-product.
Indeed, for each $i \in I$, we can choose a list $(a_i^{(1)},\dots,a_i^{(n)}) \in \End(V_i)^n$
such that $u_i=\underset{k=1}{\overset{n}{\prod}} a_i^{(k)}$ and
$p_k(a_i^{(k)})=0$ for all $k \in \lcro 1,n\rcro$; then,
for all $k \in \lcro 1,n\rcro$ we define $a^{(k)}$ as the endomorphism of $V$ such that
$(a^{(k)})_{|V_i}=a_i^{(k)}$ for all $i \in I$, and we find that $p_k(a^{(k)})=0$
and $u=\underset{k=1}{\overset{n}{\prod}} a^{(k)}$.
\end{Rem}

\subsection{On locally unipotent operators}

The proof of Theorem \ref{upper2prodtheo} is partly based upon the following result:

\begin{lemma}\label{locallynilpotentperturbationLemma}
Let $V$ be a vector space with countable dimension, and let
$u \in \End(V)$ be locally nilpotent and $\calB$ be a basis (considered as a set) that is adapted to $u$
(with respect to the polynomial $p:=t$).
Let $v \in \End(V)$. Assume that for every vector $x \in \calB$,
there is a non-zero scalar $\lambda_x$ such that
$v(x)=\lambda_x u(x)$ modulo $\Vect(u^k(x))_{k \geq 2}$.
Then $v$ is similar to $u$.
\end{lemma}

\begin{proof}
First of all, we check that $v$ is locally nilpotent. Given $x \in \calB$,
one proves by induction that $v^k(x) \in \Vect(u^i(x))_{i \geq k+1}$ for all $k \in \N$,
and it follows that $v^k(x)=0$ for some $k \in \N$. Since $\calB$ generates $V$, this is enough
to see that $v$ is locally nilpotent.

Next, we reduce the situation to the one where $\lambda_x=1$ for all $x \in \calB$, as follows. We set $\lambda_0:=1$.
Then, we set
$$\mu_x:=\prod_{k=0}^{+\infty} \lambda_{u^k(x)} \in \F \setminus \{0\}$$
for all $x \in \calB$ (which makes sense because $u^k(x)=0$ for large enough $k$).
Note that $\lambda_x \mu_{u(x)}=\mu_x$ for all $x \in \calB$.

We define $\Lambda$ as the automorphism of $V$ that maps $x$ to $\mu_x^{-1}.x$ for all $x \in \calB$.
Then one checks that $v':=\Lambda^{-1} \circ v \circ \Lambda$ satisfies
$v'(x)=u(x)$ modulo $\Vect(u^k(x))_{k \geq 2}$ for all $x \in \calB$.
Since $v'$ is similar to $v$, no generality is lost in assuming that
\begin{equation}\label{Hassumption}
\forall x \in \calB, \quad v(x)=u(x) \quad \text{modulo} \quad \Vect(u^k(x))_{k \geq 2}.
\end{equation}

We shall prove that $v$ has the same Kaplansky invariants as $u$, and the conclusion will
then follow from Kaplansky's theorem.
To do so, we start by giving a somewhat explicit computation of the $t^\alpha(V^u)$ and $t^\alpha(V^v)$
submodules.
Since $\calB$ is adapted to $u$, the restriction of $u$ to the set $X:=\calB \cup \{0\}$
is a mapping $f : X \rightarrow X$.
By transfinite induction, we define a chain of subsets $(f^\alpha(X))_\alpha$ (indexed over the ordinals)
as follows:
\begin{itemize}
\item $f^0(X)=X$;
\item For every ordinal $\alpha$ with a predecessor, $f^\alpha(X)=f(f^{\alpha-1}(X))$;
\item For every limit ordinal $\alpha$, we put $f^\alpha(X)=\underset{\beta<\alpha}{\bigcap}
f^\beta(X)$.
\end{itemize}
We also define $f^\infty(X)$ as the intersection of all the $f^\alpha(X)$ sets (the sequence $\bigl(f^\alpha(X)\bigr)_\alpha$ actually terminates).

A transfinite induction shows that
$$\forall \alpha, \; t^\alpha(V^u)=\Vect(f^\alpha(X)).$$
Next, we prove that, for every ordinal $\alpha$,
$$t^\alpha(V^v)=\Vect(f^\alpha(X)).$$
Let us start with the case $\alpha=1$.
Firstly, $v(x) \in \im u$ for all $x \in \calB$ (by assumption \eqref{Hassumption}), and hence
$\im v \subset \im u$.
Next, let $x \in \calB$. Denote by $\nu$ the least integer $k>0$ such that $u^k(x)=0$.
By downward induction, let us prove that $u^k(x) \in \im v$ for all $k \geq 1$.
Obviously $u^k(x) \in \im v$ for all $k \geq \nu$.
Now, let $k \geq 1$ be such that $u^l(x) \in \im v$ for all $l>k$.
Then $v(u^{k-1}(x))-u(u^{k-1}(x)) \in \Vect(u^l(x))_{l \geq k+1}$; combining this with the induction hypothesis, we obtain that $u^k(x) \in \im v$.
Hence, by induction $u(x) \in \im v$. Varying $x$ yields the inclusion $\im u \subset \im v$,
and we conclude that
$$\im u=\im v=\Vect(f(X)).$$

More generally, assume that, for some ordinal $\alpha$, we have $t^\alpha(V^v)=t^\alpha(V^u)=\Vect(f^\alpha(X))$.
Applying the above line of reasoning to the endomorphisms of $\Vect(f^\alpha(X))$ induced by $u$ and $v$, respectively,
we obtain that $t^{\alpha+1}(V^v)=t^{\alpha+1}(V^u)=\Vect(f^{\alpha+1}(X))$.
Let $\alpha$ be a limit ordinal for which $t^{\beta}(V^v)=t^\beta(V^u)=\Vect(f^\beta(X))$ for every ordinal $\beta<\alpha$.
Taking the intersection, we deduce that $t^{\alpha}(V^v)=t^\alpha(V^u)=\Vect(f^\alpha(X))$.

Hence, the claimed result is obtained by transfinite induction.
By taking the intersection, we also obtain
$$t^\infty(V^u)=t^\infty(V^v)=\Vect(f^\infty(X)).$$

Next, we shall use those equalities, together with the assumptions on $v$,
to prove that $u$ and $v$ have the same Kaplansky invariants.
Let $\alpha$ be an ordinal. Note that $t^\alpha(V^u)/t^{\alpha+1}(V^u)=t^\alpha(V^v)/t^{\alpha+1}(V^v)$.
For $x \in f^\alpha(X) \setminus f^{\alpha+1}(X)$, define
$[x]_\alpha$ as the class of $x$ in the quotient vector space $\Vect(f^\alpha(X))/\Vect(f^{\alpha+1}(X))$.
Then the family $\bigl([x]_\alpha\bigr)_{x \in f^\alpha(X) \setminus f^{\alpha+1}(X)}$
is a basis of $\Vect(f^\alpha(X))/\Vect(f^{\alpha+1}(X))$, i.e.\ of
$t^\alpha(V^u)/t^{\alpha+1}(V^u)$.
Thus, assumption \eqref{Hassumption} yields that the
linear map $t^\alpha(V^u)/t^{\alpha+1}(V^u) \longrightarrow t^{\alpha+1}(V^u)/t^{\alpha+2}(V^u)$
induced by $u$ is precisely the one induced by $v$.
Hence, $\kappa_\alpha(u,t)=\kappa_\alpha(v,t)$.

In order to conclude, it remains to prove that $\kappa_\infty(u,t)=\kappa_\infty(v,t)$.
For $x \in \calB$, denote by $\nu(x)$ the least integer $k>0$ such that $u^k(x)=0$ (that is, $f^k(x)=0$).
For $k \in \N^*$, denote by $\calB_k$ the subset of $f^\infty(X) \setminus \{0\}$ consisting of the vectors $x$ such that $\nu(x)=k$. Thus, $u$ maps $\calB_{k+1}$ onto $\calB_k$ for all $k \in \N^*$.
Note that both $u$ and $v$ vanish everywhere on $\calB_1$.
For $n \in \N$, set
$$E_n:=\Vect(\calB_1 \cup \cdots \cup \calB_n).$$
Note that
$$\forall n \in \N, \; (\Ker u) \cap E_n=\underset{k=1}{\overset{n}{\bigoplus}} \bigl(\Ker u \cap \Vect(\calB_k)\bigr)$$
and that
\begin{equation}\label{umodv}
\forall n \in \N \setminus \{0,1\}, \; \forall x \in E_n, \; v(x)-u(x) \in E_{n-2}.
\end{equation}

Next, we prove by induction that $v$ maps $E_n$ onto $E_{n-1}$ for all $n \in \N^*$.
First of all, \eqref{umodv} shows that $v$ maps $E_n$ into $E_{n-1}$.
Next, let $n \geq 2$ be such that $v$ maps $E_n$ onto $E_{n-1}$.
Let $x \in \Vect(\calB_{n+1})$. Then $v(x)=u(x)+z$ for some $z \in E_{n-1}$.
By induction, $z=v(y)$ for some $y \in E_n$, and hence $u(x)=v(x-y)$ and $x-y \in E_{n+1}$.
As $u(\Vect(\calB_{n+1}))=\Vect(\calB_n)$, this yields
$\Vect(\calB_n) \subset v(E_{n+1})$, and by the induction hypothesis we conclude that
$E_n \subset v(E_{n+1})$. Hence, $v(E_{n+1})=E_n$.

Finally, we shall prove that
$$\forall n \in \N, \; \dim(\Ker u \cap E_n)=\dim(\Ker v \cap E_n).$$
Fix $n \in \N \setminus \{0,1\}$, and let $x \in E_n$ belong to $\Ker v$.
We split $x=y+z$ where $y \in \Vect(\calB_n)$ and $z \in E_{n-1}$.
Then, $v(x)=u(y)$ modulo $E_{n-2}$ by \eqref{Hassumption}. Hence, $u(y) \in E_{n-2} \cap \Vect(\calB_{n-1})=\{0\}$.
Conversely, let $y \in \Vect(\calB_n) \cap \Ker u$. Then
$v(y)=u(y)$ modulo $E_{n-2}$, and hence $v(y) \in E_{n-2}$. It follows that
$v(y)=v(z)$ for some $z \in E_{n-1}$, and hence $y-z \in E_n \cap \Ker v$.
Denoting by $\pi_n$ the projection of $E_n$ on $\Vect(\calB_n)$ along $E_{n-1}$,
we sum up our recent findings by the equality
$$\pi_n(\Ker v \cap E_n)=\Ker u \cap \Vect(\calB_n).$$
Hence, by the rank theorem
\begin{align*}
\dim(\Ker v \cap E_n) & =\dim \pi_n(\Ker v \cap E_n)+\dim (\Ker v \cap E_{n-1}) \\
& =\dim \bigl(\Ker u \cap \Vect(\calB_n)\bigr)+\dim (\Ker v \cap E_{n-1}).
\end{align*}
By induction, we deduce that
$$\forall n \in \N, \; \dim(\Ker v \cap E_n)=\dim(\Ker u \cap E_n).$$
By taking the upper bounds (in cardinals), we conclude that
$$\dim(\Ker v \cap t^\infty(V^v)) = \dim(\Ker u \cap t^\infty(V^u)),$$
that is $\kappa_\infty(v,t)=\kappa_\infty(u,t)$.

Hence, the $t$-primary endomorphisms $u$ and $v$ have the same Kaplansky invariants.
By Kaplansky's theorem, they are similar.
\end{proof}

We are now ready to prove Theorem \ref{unipotentprod2theo}.

\begin{proof}[Proof of Theorem \ref{unipotentprod2theo}]
We choose a basis $\calB$ that is adapted to $u$ (we consider that $\calB$ is a set).
For $x \in \calB \cup \{0\}$, we denote by $\nu(x)$ the least integer $k\geq 0$ such that $u^k(x)=0$.
We define two endomorphisms $a$ and $b$ of $V$ as follows on the basis $\calB$: for all $x \in \calB$,
$$a(x)=\begin{cases}
x & \text{if $\nu(x)$ is odd} \\
-x+u(x) & \text{if $\nu(x)$ is even}
\end{cases} \quad \text{and} \quad
b(x)=\begin{cases}
x+u(x) & \text{if $\nu(x)$ is odd} \\
-x & \text{if $\nu(x)$ is even.}
\end{cases}$$
Noting that $\nu(u(x))=\nu(x)-1$ for all $x \in \calB$, one easily checks that $a^2=\id_V$ and $b^2=\id_V$.
Moreover, for all $x \in \calB$,
$$(ab)(x)=\begin{cases}
x-u(x)+u^2(x) & \text{if $\nu(x)$ is odd} \\
x-u(x) & \text{if $\nu(x)$ is even.}
\end{cases}$$
Thus $v:=ab-\id_V$ satisfies the conditions in Lemma \ref{locallynilpotentperturbationLemma},
and hence it is similar to $u$. It follows that $\id_V+v$ is similar to $\id_V+u$, and hence
$\id_V+u$ is similar to a $(t^2-1,t^2-1)$-product. Hence, $\id_V+u$ is a $(t^2-1,t^2-1)$-product.

The proof that $\id_V+u$ is the product of two unipotent endomorphisms of index $2$ is essentially similar, the only difference
being that $a$ and $b$ are now defined as follows: for all $x \in \calB$,
$$a(x)=\begin{cases}
x & \text{if $\nu(x)$ is odd} \\
x+u(x) & \text{if $\nu(x)$ is even}
\end{cases} \quad \text{and} \quad
b(x)=\begin{cases}
x+u(x) & \text{if $\nu(x)$ is odd} \\
x & \text{if $\nu(x)$ is even.}
\end{cases}$$
It is easily seen that $a$ and $b$ are unipotent of index $2$, and one uses Lemma \ref{locallynilpotentperturbationLemma}
to obtain that $\id_V+u$ is similar to $ab$. We conclude that $\id_V+u$ is the product of two unipotent endomorphisms of index $2$.
\end{proof}

\subsection{Results on block operators}

\begin{lemma}\label{Llemma1}
Let $\calA$ be an $\F$-algebra, and let $a \in \calA$.
Set $b:=\begin{bmatrix}
0 & -1_\calA \\
1_\calA & a
\end{bmatrix} \in \Mat_2(\calA)$.
Then, in the algebra $\Mat_2(\calA)$, the element $b$ is the product of two involutions, and it is
also the product of two unipotent elements of index $2$.
\end{lemma}

\begin{proof}
Set $c:=2\,1_\calA-a$. Let $\varepsilon \in \{1,-1\}$.
Then
$$\begin{bmatrix}
1_\calA & -\varepsilon c \\
0 & \varepsilon\,1_\calA
\end{bmatrix}
\,
\begin{bmatrix}
1_\calA & 0 \\
\varepsilon\,1_\calA & \varepsilon\,1_\calA
\end{bmatrix}=\begin{bmatrix}
a-1_\calA & a-2\,1_\calA \\
1_\calA & 1_\calA
\end{bmatrix}.$$
Setting
$$T:=\begin{bmatrix}
1_\calA & 1_\calA-a \\
0 & 1_\calA
\end{bmatrix},$$
we see that $T$ is invertible and that
$$T\begin{bmatrix}
a-1_\calA & a-2\,1_\calA \\
1_\calA & 1_\calA
\end{bmatrix}T^{-1}=b.$$
One checks that both $\begin{bmatrix}
1_\calA & -\varepsilon\,c \\
0 & \varepsilon\,1_\calA
\end{bmatrix}$ and $\begin{bmatrix}
1_\calA & 0 \\
\varepsilon\,1_\calA & \varepsilon\,1_\calA
\end{bmatrix}$ are annihilated by the polynomial $(t-1)(t-\varepsilon)$.
Taking $\varepsilon=-1$ (respectively, $\varepsilon=1$), we deduce that
$b$ is conjugated to the product of two involutions (respectively, of two unipotent elements of index $2$),
and hence it is the product of two involutions (respectively, of two unipotent elements of index $2$).
\end{proof}

\begin{Not}
Let $p \in \F[t]$ be monic with degree $d$. For $\delta \in \F \setminus \{0\}$, we set
$$R_\delta(p):=t^d p(t+\delta t^{-1}),$$
which is a monic polynomial of degree $2d$.
\end{Not}

Note that $R_\delta(p)(0) \neq 0$.
One checks that, given monic polynomials $p$ and $q$,
$$R_\delta(pq)=R_\delta(p)\,R_\delta(q),$$

In particular if $R_\delta(p)$ is irreducible then so is $p$.

Note that $R_1(p)$ is a quasi-palindromial for every monic polynomial $p \in \F[t]$.
A sort of converse statement holds (it is folklore and we do not give any detail of its proof):

\begin{lemma}\label{R1lemma}
Let $p \in \F[t]$ be a quasi-palindromial with no root in $\{1,-1\}$.
Then $p=R_1(q)$ for some monic polynomial $q \in \F[t]$.
\end{lemma}

\begin{cor}\label{R1cor}
Let $p \in \F[t] \in \Irr(\F) \setminus \{t+1,t-1\}$ be such that $p=p^\sharp$.
Then $p=R_1(q)$ for some $q \in \Irr(\F) \setminus \{t+2,t-2\}$.
\end{cor}

\begin{proof}[Proof of Corollary \ref{R1cor}]
Lemma \ref{R1lemma} yields a monic polynomial $q \in \F[t]$ such that $R_1(q)=p$.
Using one of the above remarks, we find that $q$ is irreducible.
Noting that $R_1(t+2)=(t+1)^2$ and $R_1(t-2)=(t-1)^2$, we conclude that $q \not\in \{t+2,t-2\}$.
\end{proof}

%
%
%

\begin{lemma}\label{Llemma2}
Let $p \in \Irr(\F) \setminus \{t-2,t+2\}$, and let $V$ be a vector space.
Assume that $q:=R_1(p)$ is irreducible and let $a \in \End(V)$ be $p$-primary.
Then $L(a)$ is $q$-primary and, for every $\alpha$, either an ordinal or $\infty$,
$$\kappa_\alpha(L(a),q)=\kappa_\alpha(a,p).$$
\end{lemma}

\begin{proof}
Set $b:=L(a)$ and denote by $d$ the degree of $p$.
First of all, we find that
$$M\bigl(b^{-1}\bigr)=\begin{bmatrix}
a & \id_V \\
-\id_V & 0
\end{bmatrix}$$
and hence
$$M\bigl(b+b^{-1}\bigr)=\begin{bmatrix}
a & 0 \\
0 & a
\end{bmatrix}.$$
It follows that
$$M\bigl(q(b)\bigr)=\begin{bmatrix}
p(a) & 0 \\
0 & p(a)
\end{bmatrix}\, M(b^d)=M(b^d)\,\begin{bmatrix}
p(a) & 0 \\
0 & p(a)
\end{bmatrix}\,$$
and, more generally,
$$\forall n \in \N, \; M(q(b)^n)=\begin{bmatrix}
p(a)^n & 0 \\
0 & p(a)^n
\end{bmatrix}\,M(b^{nd})=M(b^{nd})\,\begin{bmatrix}
p(a)^n & 0 \\
0 & p(a)^n
\end{bmatrix}.$$
Since $a$ is $p$-primary, it is clear from the last point that $L(a)$ is $q$-primary.

Next, since $b$ is $q$-primary and invertible,
we see that $b$ induces an automorphism of every linear subspace it stabilizes, in particular it maps bijectively
$q^\alpha((V^2)^{b})$ onto itself for every ordinal $\alpha$ as well as for $\alpha=\infty$.
Then we prove by transfinite induction that for every ordinal $\alpha$,
$$q^\alpha\bigl((V^2)^{b}\bigr)=\bigl(p^\alpha(V^a)\bigr)^2.$$
Indeed, if this holds for some ordinal $\alpha$, then $b^{d}$ maps $\bigl(p^\alpha(V^a)\bigr)^2$
onto itself, and hence the above description of $q(b)$ shows that
$$q(b)\bigl[q^\alpha\bigl((V^2)^{b}\bigr)\bigr]=\bigl(p^{\alpha+1}(V^a)\bigr)^2,$$
that is
$$q^{\alpha+1}\bigl((V^2)^{b}\bigr)=\bigl(p^{\alpha+1}(V^a)\bigr)^2.$$
Moreover, if for some ordinal $\alpha$ with no predecessor, we have
$$\forall \beta<\alpha, \; q^\beta\bigl((V^2)^{b}\bigr)=\bigl(p^\beta(V^a)\bigr)^2,$$
then, by taking the intersection we deduce that
$$q^\alpha\bigl((V^2)^{b}\bigr)=\bigl(p^\alpha(V^a)\bigr)^2.$$
Hence, this holds for every ordinal $\alpha$, and by taking the full intersection we obtain that it holds for $\alpha=\infty$ as well.

Next, let $\alpha$ be an ordinal.
The linear map from
$q^\alpha\bigl((V^2)^{b}\bigr)/q^{\alpha+1}\bigl((V^2)^{b}\bigr)$ to
$q^{\alpha+1}\bigl((V^2)^{b}\bigr)/q^{\alpha+2}\bigl((V^2)^{b}\bigr)$ induced by $q(b)$
equals the composite of the linear map
$$\overline{u} : q^\alpha\bigl((V^2)^{b}\bigr)/q^{\alpha+1}\bigl((V^2)^{b}\bigr) \longrightarrow q^{\alpha+1}\bigl((V^2)^{b}\bigr)/q^{\alpha+2}\bigl((V^2)^{b}\bigr)$$
induced by $p(b+b^{-1})$ with the automorphism of $q^{\alpha+1}\bigl((V^2)^{b}\bigr)/q^{\alpha+2}\bigl((V^2)^{b}\bigr)$ induced by $b^d$.
Hence, its kernel equals the one of $\overline{u}$, and by the above expression of $p(b+b^{-1})$ we deduce that
$\overline{u}$ is equivalent to the direct product of two copies of the $\F$-linear mapping
$$\overline{v}:=p^\alpha(V^a)/p^{\alpha+1}(V^a) \longrightarrow p^{\alpha+1}(V^a)/p^{\alpha+2}(V^a)$$
induced by $p(a)$. Since $\deg q=2d$, this yields
$$2\,d\, \kappa_\alpha(b,q)=\dim_\F \Ker \overline{u}=2 \dim_\F \Ker \overline{v}=2\,d\,\kappa_\alpha(a,p),$$
and we conclude that $\kappa_\alpha(b,q)=\kappa_\alpha(a,q)$.
The equality $\kappa_\infty(b,q)=\kappa_\infty(a,q)$ is obtained in a similar fashion.
\end{proof}

\subsection{Proof of Theorem \ref{2involinfinitetheo}}

Let $V$ be a vector space with countable dimension, and let $u \in \GL(V)$ be locally finite.
We already know that if $u$ is the product of two involutions then $u$ is similar to its inverse.
Moreover, we know from Corollary \ref{simtoinverseCor} that conditions (ii) and (iii) are equivalent in Theorem
\ref{2involinfinitetheo}. In order to conclude, it suffices to prove that condition (iii) implies condition (i).

Hence, let us assume that $\kappa_\alpha(u,p)=\kappa_\alpha(u,p^\sharp)$
for every $\alpha \in \omega_1 \cup \{\infty\}$ and every $p \in \Irr(\F)\setminus \{t\}$.
We wish to prove that $u$ is the product of two involutions.

First of all, $u_{t-1}$ is the sum of the identity of $\Tor_{(t-1)^\infty}(u)$
and of a locally nilpotent endomorphism of it. By Theorem \ref{unipotentprod2theo}, it is the product of two involutions.
Likewise, $-u_{t+1}=ab$ for some involutions $a$ and $b$ of $\Tor_{(t+1)^\infty}(u)$, and hence $u_{t+1}=(-a)b$ is the product of two involutions.

Let $p \in \Irr(\F) \setminus \{t,t-1,t+1\}$ be a quasi-palindromial. By Corollary \ref{R1cor}, $p=R_1(q)$ for some
$q \in \Irr(\F) \setminus \{t+2,t-2\}$. The sequence $(\kappa_\alpha(u,p))_{\alpha \in \omega_1 \cup \{\infty\}}$
is admissible, and hence there exists a $q$-primary endomorphism $w$ of a vector space $W$ with countable dimension
such that $\kappa_\alpha(w,q)=\kappa_\alpha(u,p)$ for all $\alpha \in \omega_1 \cup \{\infty\}$.
By Lemma \ref{Llemma1}, $L(w)$ is $p$-primary and $\kappa_\alpha(L(w),p)=\kappa_\alpha(w,q)$ for all $\alpha \in \omega_1 \cup \{\infty\}$.
Hence, Kaplansky's theorem shows that $L(w)$ is similar to $u_p$.
Besides, Lemma \ref{Llemma2} shows that $L(w)$ is the product of two involutions, and we conclude that so is $u_p$.

Finally, let $p \in \Irr(\F) \setminus \{t\}$ be such that $p^\sharp \neq p$. In particular $p \not\in \{t-1,t+1\}$.
Combining our assumptions with Corollary \ref{simtoinverseCor},
we find that
$$\forall \alpha \in \omega_1 \cup \{\infty\}, \; \kappa_\alpha\bigl(u_{p^\sharp}^{-1},p\bigr)=\kappa_\alpha(u_{p^\sharp},p^\sharp)
=\kappa_\alpha(u,p^\sharp)=\kappa_\alpha(u,p)=\kappa_\alpha(u_p,p),$$
and hence $u_{p^\sharp}^{-1}$ is similar to $u_p$. We deduce that $u_{p^\sharp}$ is similar to $(u_p)^{-1}$, and
we conclude that the endomorphism $v_p$ of $\Tor_{p^\infty}(u) \oplus \Tor_{(p^\sharp)^\infty}(u)$
induced by $u$ is similar to the endomorphism $w$ of $\bigl(\Tor_{p^\infty}(u)\bigr)^2$ such that
$$M(w)=\begin{bmatrix}
u_p & 0 \\
0 & (u_p)^{-1}
\end{bmatrix}.$$
Lemma \ref{inversediagonalLemma} shows that $w$ is the product of two involutions, and we conclude that so is $v_p$.

Splitting $\Irr(\F) \setminus \{t\}$ into the orbits of the involution $p \in \Irr(\F) \setminus \{t\} \mapsto p^\sharp$, we conclude by Remark
\ref{superdirectsumremark} that $u$ is the product of two involutions.

\subsection{Proof of Theorem \ref{upper2prodbistheo}}

Let $M \in \calM_{Cf}(\F)$ be an upper-triangular matrix whose diagonal entries equal $1$ or $-1$.
Denote by $u$ an associated endomorphism of a vector space $V$ over $\F$.
Let $p \in \Irr(\F) \setminus \{t,t-1,t+1\}$. Then $p(1) \neq 0$ and $p(-1) \neq 0$, and hence
$p(M)$ is upper-triangular with all diagonal entries non-zero. It follows that $p(u)$ is an automorphism of $V$,
and hence $\kappa_\alpha(u,p)=0$ for all $\alpha \in \omega_1 \cup \{\infty\}$.
Noting that $(t-1)^\sharp=t-1$ and $(t+1)^\sharp=t+1$, we deduce that $u$ satisfies condition (iii)
from Theorem \ref{2involinfinitetheo}. Hence, $u$ is the product of two involutory automorphisms of $V$, and
we conclude that $M$ is the product of two involutions in the group $\GL_{Cf}(\F)$.

\section{Products of two unipotent endomorphisms of index $2$}\label{2unipotentSection}

For fields of characteristic $2$, unipotent elements of index $2$ coincide with involutions.
Throughout this section, we will assume that the characteristic of $\F$ differs from $2$, and in that case we will characterize the 
automorphisms that are the product of two unipotent automorphisms of index $2$.

\subsection{Main result}

Our aim here is to prove the following theorem.

\begin{theo}\label{2unipotentTheo}
Let $V$ be a vector space with countable dimension. Assume that $\car(\F) \neq 2$.
Let $u \in \GL(V)$ be locally finite. The following conditions are equivalent:
\begin{enumerate}[(i)]
\item $u$ is the product of two unipotent endomorphisms of index $2$.
\item $u$ is similar to its inverse, and $\kappa_\alpha(u,t+1)=0$ for every even ordinal $\alpha \in \omega_1$.
\item For every $p \in \Irr(\F) \setminus \{t\}$, one has
$\forall \alpha \in \omega_1 \cup \{\infty\}, \; \kappa_\alpha(u,p)=\kappa_\alpha(u,p^\sharp)$, and
$\kappa_\alpha(u,t+1)=0$ for every even ordinal $\alpha \in \omega_1$.
\end{enumerate}
\end{theo}

Here, the equivalence between conditions (ii) and (iii) is a straightforward consequence of
Corollary \ref{simtoinverseCor}. The difficulty lies in the equivalence between conditions (i) and (iii).
We will actually prove that condition (iii) implies condition (i) (Section \ref{iiiimpliesi2unipotentSection}) and that
condition (i) implies condition (ii) (Section \ref{iimpliesii2unipotentSection}).

\subsection{A lemma on block operators}

\begin{lemma}\label{Llemma3}
Set $p:=t+2$ and let $a \in \End(V)$.
Then $L(a)$ is $(t+1)$-primary if and only if $a$ is $(t+2)$-primary. Moreover:
\begin{enumerate}[(i)]
\item For every ordinal $\delta$ with no predecessor and every integer $k$,
$$\kappa_{\delta+2k+1}(L(a),t+1)=\kappa_{\delta+k}(a,t+2) \quad \text{and} \quad
\kappa_{\delta+2k}(L(a),t+1)=0.$$

\item One has
$$\kappa_{\infty}(L(a),t+1)=\kappa_\infty(a,t+2).$$
\end{enumerate}
\end{lemma}

\begin{proof}
The strategy of proof is quite similar to the one of Lemma \ref{Klemma2}.

Set $b:=L(a)$.
We note that
$$(b+\id)^2=b\,c=c\,b$$
where
$$M(c):=\begin{bmatrix}
a+2\id_V & 0 \\
0 & a+2\id_V
\end{bmatrix}.$$
Noting that $R_1(t+2)=(t+1)^2$, one uses a similar line of reasoning as in the proof of Lemma \ref{Llemma2}
to obtain that $L(a)$ is $(t+1)$-primary if and only if $a$ is $(t+2)$-primary. Moreover, one proves by transfinite induction that for every
ordinal $\delta$ with no predecessor and every $k \in \N$,
$$(t+1)^{\delta+2k}\bigl((V^2)^b\bigr)=\bigl((t+2)^{\delta+k}(V^a)\bigr)^2.$$
By taking the intersection, we also obtain that
$$(t+1)^\infty\bigl((V^2)^b\bigr)=\bigl((t+2)^{\infty}(V^a)\bigr)^2.$$

Next, we compute $\kappa_\infty(b,t+1)$.
Since $b+\id$ induces a surjective endomorphism of $(t+1)^\infty\bigl((V^2)^b\bigr)$,
it induces a surjective linear map from
$\Ker (b+\id)^2 \cap (t+1)^\infty\bigl((V^2)^b\bigr)$ to $\Ker (b+\id) \cap (t+1)^\infty\bigl((V^2)^b\bigr)$ whose kernel equals
$\Ker (b+\id) \cap (t+1)^\infty\bigl((V^2)^b\bigr)$, and it follows that
\begin{align*}
\dim_\F\Bigl[\Ker (b+\id)^2 \cap (t+1)^\infty\bigl((V^2)^b\bigr)\Bigr]
& =2 \dim_\F \Bigl[\Ker (b+\id) \cap (t+1)^\infty\bigl((V^2)^b\bigr)\Bigr] \\
& =2\, \kappa_\infty(b,t+1).
\end{align*}
Clearly, the above expression of $(b+\id)^2$ shows that
$$\Ker (b+\id)^2 \cap  (t+1)^\infty\bigl((V^2)^b\bigr)=\bigl(\Ker (a+2\id_V) \cap (t+2)^\infty(V^a)\bigr)^2$$
and hence
\begin{align*}
\dim_\F\Bigl[\Ker (b+\id)^2 \cap (t+1)^\infty\bigl((V^2)^b\bigr)\Bigr]
&= 2 \dim_\F \Bigl[\Ker (a+2\id_V) \cap (t+2)^\infty(V^a) \Bigr] \\
& =2\, \kappa_\infty(a,t+2).
\end{align*}
It follows that $\kappa_\infty(b,t+1)=\kappa_\infty(a,t+2)$.

Next, we compute the remaining Kaplansky invariants of $b$ as functions of those of $a$.
First of all, let $\alpha$ be an ordinal.
By Lemma \ref{supersumdimLemma}, the dimension over $\F$ of the kernel of the linear map
$$(t+1)^\alpha\bigl((V^2)^b\bigr)/(t+1)^{\alpha+2}\bigl((V^2)^b\bigr) \longrightarrow
(t+1)^{\alpha+2}\bigl((V^2)^b\bigr)/(t+1)^{\alpha+4}\bigl((V^2)^b\bigr)$$
induced by $(b+\id)^2$ equals
$$\kappa_{\alpha}(b,t+1)+2\, \kappa_{\alpha+1}(b,t+1)+\kappa_{\alpha+2}(b,t+1).$$

Next, let $\delta$ be an ordinal with no predecessor, and $k \in \N$. Then, for all $i \in \{0,1,2\}$, 
$$(t+1)^{\delta+2k+2i}\bigl((V^2)^b\bigr)=\bigl((t+2)^{\delta+k+i}(V^a)\bigr)^2.$$
The surjective linear mapping
$$(t+1)^{\delta+2k}\bigl((V^2)^b\bigr)/(t+1)^{\delta+2k+2}\bigl((V^2)^b\bigr)
\longrightarrow
(t+1)^{\delta+2k+2}\bigl((V^2)^b\bigr)/(t+1)^{\delta+2k+4}\bigl((V^2)^b\bigr)$$
induced by $(b+\id)^2$ has its kernel equal to the one of the surjective linear mapping induced by
$c$, which is equivalent to the direct sum of two copies of the surjective linear mapping
$$(t+2)^{\delta+k}(V^a)/(t+2)^{\delta+k+1}(V^a) \longrightarrow (t+2)^{\delta+k+1}(V^a)/(t+2)^{\delta+k+2}(V^a)$$
induced by $a+2\id_V$.
From this, we deduce the identity
$$\kappa_{\delta+2k}(b,t+1)+2\, \kappa_{\delta+2k+1}(b,t+1)+\kappa_{\delta+2k+2}(b,t+1)=2\,\kappa_{\delta+k}(a,t+2).$$

Hence, in order to conclude it remains to prove that
$\kappa_{\delta+2k}(b,t+1)=0$ for every ordinal $\delta$ with no predecessor and every $k\in \N$.
Just like in the proof of Lemma \ref{Klemma2}, it suffices to do so in the special case when $\delta=k=0$.

Thus, we need to prove that the linear mapping
$$V/(b+\id)((V^2)^b) \longrightarrow (b+\id)((V^2)^b)/(b+\id)^2((V^2)^b)$$
induced by $b+\id$ is injective.

Let $(x,y) \in V^2$ be such that $(b+\id)(x,y) \in \im (b+\id)^2$. We need to prove that
$(x,y) \in \im (b+\id)$. Set $a':=a+2\id_V$. The assumption on $(x,y)$ yields that $x-y=a'(z)$ for some $z \in V$.
Setting $y_0:=-z$, and $x_0:=y-a'(y_0)+y_0=y+a'(z)-z=x-z$,
we gather that $x=x_0-y_0$ and $y=x-a'(z)=x_0+a'(y_0)-y_0$, that is
$(x,y)=(b+\id)(x_0,y_0)$. Thus, the claimed statement is proved.
We conclude that $\kappa_{0}(b,t+1)=0$.

We deduce that $\kappa_{\delta+2k}(b,t+1)=0$ for every ordinal $\delta$ with no predecessor and every $k\in \N$, which completes the proof.
\end{proof}

\subsection{Proof of (iii) $\Rightarrow$ (i) in Theorem \ref{2unipotentTheo}}
\label{iiiimpliesi2unipotentSection}

Let $V$ be a vector space with countable dimension, and let $u \in \GL(V)$ be locally finite.
Assume that for every $p \in \Irr(\F) \setminus \{t\}$ and every
$\alpha \in \omega_1 \cup \{\infty\}$, $\kappa_\alpha(u,p)=\kappa_\alpha(u,p^\sharp)$,
and that for every even ordinal $\alpha \in \omega_1$, $\kappa_\alpha(u,t+1)=0$.
By Theorem \ref{unipotentprod2theo}, $u_{t-1}$ is the product of two unipotent endomorphisms of index $2$.
Let $p \in \Irr(\F) \setminus \{t\}$ be such that $p \neq p^\sharp$. In particular $p \not\in \{t-1,t+1\}$.
With the same line of reasoning as in the proof of Theorem \ref{2involinfinitetheo}, one finds that
the endomorphism $v_p$ of $\Tor_{p^\infty}(u) \oplus \Tor_{(p^\sharp)^\infty}(u)$
induced by $u$ is similar to the endomorphism $w$ of $\bigl(\Tor_{p^\infty}(u)\bigr)^2$
such that $M(w)=\begin{bmatrix}
u_p & 0 \\
0 & (u_p)^{-1}
\end{bmatrix}$.
Because $p$ is coprime with both $t+1$ and $t-1$, we see that $w^2-\id=(w-\id)(w+\id)$ is an automorphism.
Hence, by Lemma \ref{inversediagonalLemma}, the matrix $M(w)$ is the product of two unipotent elements of index $2$ of $\Mat_2\bigl(\End(\Tor_{p^\infty}(u))\bigr)$,
and hence $v_p$ is the product of two unipotent endomorphisms of index $2$ of $\Tor_{p^\infty}(u) \oplus \Tor_{(p^\sharp)^\infty}(u)$.

Let $p \in \Irr(\F) \setminus \{t-1,t+1\}$ be such that $p^\sharp=p$. Then $p=R_1(q)$ for some $q \in \Irr(\F) \setminus \{t-2,t+2\}$.
With the same line of reasoning as in the proof of Theorem \ref{2involinfinitetheo}, one finds a vector space $W$ with countable dimension and a $q$-primary endomorphism $w$ of $W$ such that $u_p$ is similar to $L(w)$. Lemma \ref{Llemma1} shows that $L(w)$ is the product of two unipotent endomorphisms of index $2$ of $W^2$,
and hence $u_p$ is the product of two unipotent endomorphisms of index $2$.

Finally, for every $\alpha \in \omega_1$, write $\alpha=\delta+k$ for some ordinal $\delta$ with no predecessor and some $k \in \N$, and
set
$$m_\alpha:=\kappa_{\delta+2k+1}(u,t+1).$$
Set also
$$m_\infty:=\kappa_\infty(u,t+1).$$
Clearly, $(m_\alpha)_{\alpha \in \omega_1 \cup \{\infty\}}$ is an admissible family of cardinals, and hence
there is a vector space $Z$ with countable dimension together with a $(t+2)$-primary endomorphism $z$ of $Z$
such that $m_\alpha=\kappa_\alpha(z,t+2)$ for all $\alpha \in \omega_1 \cup \{\infty\}$.
By combining Lemma \ref{Llemma3} with Kaplansky's theorem, we find that
$u_{t+1}$ is similar to $L(z)$. By Lemma \ref{Llemma1} the endomorphism $L(z)$ is the product of two unipotent endomorphisms of index $2$,
and hence so is $u_{t+1}$.

Since $p \in \Irr(\F) \mapsto p^\sharp \in \Irr(\F)$ is an involution, we
combine the above results with Remark \ref{superdirectsumremark} to conclude that $u$ is the product of  two unipotent endomorphisms of index $2$.

\subsection{Proof of (i) $\Rightarrow$ (ii) in Theorem \ref{2unipotentTheo}}
\label{iimpliesii2unipotentSection}

Let $V$ be a vector space with countable dimension, and let $u \in \GL(V)$ be locally finite.
Assume that $u=ab$ for some endomorphisms $a,b$ of $V$ such that $(a-\id_V)^2=(b-\id_V)^2=0$.

First of all, we reduce the situation to the one where $1$ is not an eigenvalue of $u$.
Let $p \in \Irr(\F) \setminus \{t-1\}$.
By Corollary \ref{commutationCor2}, $a$ and $b$ commute with $u+u^{-1}$ and hence
they stabilize $\Ker R_1(q)[u]=\Ker q(u+u^{-1})$ for each monic polynomial $q \in \F[t]$.
If $p=p^\sharp$ then $p=R_1(q)$ for some $q \in \Irr(\F)$; by the above we see that
$a$ and $b$ stabilize $\Ker R_1(q^n)=\Tor_{p^n}(u)$ for all $n \in \N$, and we conclude that
$a$ and $b$ stabilize $\Tor_{p^\infty}(u)$. In particular, they stabilize $\Tor_{(t+1)^\infty}(u)$
and hence $u_{t+1}$ is a $\bigl((t-1)^2,(t-1)^2\bigr)$-product.

Assume now that $p \neq p^\sharp$. Then $pp^\sharp=R_1(q)$ for some monic $q \in \F[t]$.
This time around, we see that $a$ and $b$ stabilize $\Ker R_1(q^n)=\Tor_{p^n}(u)\oplus \Tor_{(p^\sharp)^n}(u)$
for all $n \in \N$, and hence they stabilize $\Tor_{p^\infty}(u)\oplus \Tor_{(p^\sharp)^\infty}(u)$.

Finally, $a$ and $b$ stabilize $\Ker \bigl((u+u^{-1})^k\bigr)=\Ker (u+\id_V)^{2k}$ for all $k \in \N$,
and hence they stabilize $\Tor_{(t+1)^\infty}(u)$.

It follows that $a$ and $b$ stabilize $W:=\underset{p \in \Irr(\F) \setminus \{t-1\}}{\bigoplus} \Tor_{p^\infty}(u)$.
The induced endomorphisms $a_W$ and $b_W$ are unipotent of index $2$, and $a_Wb_W$ is the
endomorphism of $W$ induced by $u$. Hence, in order to prove that condition (ii) holds,
it suffices to do so under the additional provision that $u-\id_V$ is invertible.

So, from now on we assume that $u-\id_V$ is invertible.
Let us write $u-\id_V=ab-\id_V=a(b-a^{-1})$ and note that $b-\id_V$ and $-(a^{-1}-\id_V)$ have square zero.
Hence, $b-a^{-1}=(b-\id_V)-(a^{-1}-\id_V)$ is both invertible and the sum of two square-zero endomorphisms.
Setting
$$V_0:=\Ker(a^{-1}-\id_V)=\Ker(a-\id_V) \quad \text{and} \quad V_1:=\Ker(b-\id_V),$$
we deduce from Lemma \ref{sum2invlemma}
that $V=V_0 \oplus V_1$, that $a-\id_V=a(\id_V-a^{-1})$ induces an isomorphism $u_1 : \Ker(b-\id_V) \overset{\simeq}{\rightarrow} \Ker(a-\id_V)$ (because $a$ induces an automorphism of $\Ker(a-\id_V)$), and
$b-\id_V$ induces an isomorphism $u_2 : \Ker(a-\id_V) \overset{\simeq}{\rightarrow}  \Ker(b-\id_V)$.
Then,
$$\varphi :
\begin{cases}
(V_0)^2 & \longrightarrow V \\
(x,y) & \longmapsto x+u_2(y)
\end{cases}$$
is an isomorphism of vector spaces and
$$M(\varphi^{-1} \circ (a-\id_V) \circ \varphi)=\begin{bmatrix}
0 & u_1u_2 \\
0 & 0
\end{bmatrix}
\quad \text{and} \quad
M(\varphi^{-1} \circ (b-\id_V) \circ \varphi)=\begin{bmatrix}
0 & 0 \\
\id_{V_0} & 0
\end{bmatrix}.$$
Setting $c:=u_1 u_2 \in \GL(V_0)$, we deduce that
$$M(\varphi^{-1} \circ u \circ \varphi)
=\begin{bmatrix}
\id_{V_0} & c \\
0 & \id_{V_0}
\end{bmatrix}
\,\begin{bmatrix}
\id_{V_0} & 0 \\
\id_{V_0} & \id_{V_0}
\end{bmatrix}=\begin{bmatrix}
\id_{V_0}+c & c \\
\id_{V_0} & \id_{V_0}
\end{bmatrix}.$$
The matrix
$$T:=\begin{bmatrix}
\id_{V_0} & -(\id_{V_0}+c) \\
0 & \id_{V_0}
\end{bmatrix} \in \Mat_2(\End(V_0))$$
is obviously invertible and
$$T\, M(\varphi^{-1} \circ u \circ \varphi)\, T^{-1}
=\begin{bmatrix}
0 & -\id_{V_0} \\
\id_{V_0} & 2\id_{V_0}+c
\end{bmatrix}.$$
Hence, $u$ is similar to $L(2\id_{V_0}+c)$.
By Lemma \ref{Llemma1}, we deduce that $u$ is the product of two involutions, and hence it is
similar to its inverse.

It remains to consider the Kaplansky invariants of $u$ associated with the polynomial $t+1$.
Since $u_{t+1}$ is a $\bigl((t-1)^2,(t-1)^2\bigr)$-product, the above applies to it, which shows that
$u_{t+1}$ is similar to $L(d)$ for some endomorphism $d$ of a vector space with countable dimension. Hence, by Lemma \ref{Llemma3},
$\kappa_{\alpha}(u,t+1)=\kappa_{\alpha}(u_{t+1},t+1)=0$ for every even ordinal $\alpha \in \omega_1$.
We conclude that $u$ satisfies condition (ii) in Theorem \ref{2unipotentTheo}.

This completes the proof of Theorem \ref{2unipotentTheo}.

\appendix

\section{Appendix. On decomposing a locally nilpotent endomorphism into the sum of two square-zero ones}

Theorem \ref{unipotentprod2theo} was claimed in \cite{Slowik}, but later the author retracted
that claim. The incorrect proof that appeared in \cite{Slowik} relied upon a lemma (lemma 2.1 there)
that can be restated as follows:

\begin{quote}
Every strictly upper-triangular matrix of $\calM_{Cf}(\F)$ is conjugated (in the ring $\calM_{Cf}(\F)$)
to some matrix in which every row and column contains at most one non-zero coefficient.
\end{quote}

Here, we shall prove that this statement is wrong. To this end we need several additional definitions. Let $u$ be an endomorphism of a vector space $V$.
We define the \textbf{core} of $u$ as $\Co(u):=\underset{n \in\N}{\bigcap} u^n(V)$. \\
We say that $u$ is a \textbf{finite Jordan cell} with size $n$ if there exists a basis $(e_1,\dots,e_n)$ of $V$
such that $u(e_i)=e_{i-1}$ for all $i \in \lcro 2,n\rcro$, and $u(e_1)=0$.
In that case $\Co(u)=\{0\}$. \\
We say that $u$ is an \textbf{infinite Jordan cell} if there exists a basis $(e_n)_{n \in \N}$ of $V$
such that $u(e_n)=e_{n-1}$ for all $n \in \N \setminus \{0\}$, and $u(e_0)=0$. In that case $\Co(u)=V$.

Now, we recall the following basic result:

\begin{prop}\label{coeur}
Assume that we have a splitting of $u$ into a (potentially infinite) direct sum of Jordan cells (finite or not): in other words $V=\underset{i \in I}{\bigoplus} V_i$ where each subspace $V_i$ is stable under $u$
and the resulting endomorphism $u_i$ is a Jordan cell. Then $u(\Co(u))=\Co(u)$.
\end{prop}

\begin{proof}
We obviously have
$\Co(u)=\underset{i \in I}{\bigoplus} \Co(u_i)=\underset{i \in J}{\bigoplus} V_i$ where
$J$ is the set of all $i \in I$ such that $V_i$ is infinite-dimensional.
We conclude that
$$u(\Co(u))=\underset{i \in J}{\sum} u(V_i)=\underset{i \in J}{\sum} V_i=\Co(u).$$
\end{proof}

From there, we reprove the following known result:

\begin{prop}\label{contrex}
There exist a vector space $V$ with countable dimension and an endomorphism $u$ of $V$
that is locally nilpotent and yet is not a direct sum of Jordan cells.
\end{prop}

\begin{proof}
Let $s,t$ be independent indeterminates, and
consider the subspace $V$ of $\F[s,t]$
spanned by $1=s^0t^0$ and the monomials of the form $s^it^j$ where $0 \leq i<j$.
Then, consider the endomorphism $u$ of $V$ defined on that basis as follows:
$u(1)=0$, $u(s^it^j)=s^{i-1}t^j$ for all $0<i<j$, and $u(t^j)=1$ for all $j \in \N \setminus \{0\}$.
It is locally nilpotent because on the one hand $u(1)=0$, and on the other hand
$u^{i+2}(s^it^j)=0$ for all $0 \leq i<j$.
One sees that for all $n \in\N$, the range of $u^n$
is spanned by $1$ and the monomials $s^it^j$ such that $0 \leq i < j-n$.
It follows that $\Co(u)=\Span(1)$, to the effect that $u(\Co(u))=\{0\} \subsetneq \Co(u)$.
Hence, Proposition \ref{coeur} shows that $u$ is not a direct sum of Jordan cells.
\end{proof}

Now, we can turn to the connection of the above with lemma 2.1 from \cite{Slowik}.
Here is a restatement of its first part:

\begin{quote}
Let $V$ be a vector space with infinite countable dimension, and let $u \in \End(V)$ be locally nilpotent.
Then $u$ is represented, in some basis of $V$, by a matrix
in which every row and column contains at most one non-zero coefficient.
\end{quote}

Note that the matrix from the conclusion is not claimed to be strictly upper-triangular.

We will show that the above statement is false, because of the following result:

\begin{prop}\label{decompJordan}
Let $u$ be an endomorphism of a vector space $V$ equipped with a basis $(e_n)_{n \in \N}$.
Assume that $u$ is locally nilpotent and that the matrix of $u$ in $(e_n)_{n \in \N}$
contains at most one non-zero coefficient in each one of its rows and columns. Then $u$
is a direct sum of Jordan cells.
\end{prop}

Hence, the logical consequence of lemma 2.1 from \cite{Slowik} would be that every locally nilpotent endomorphism
of a countable-dimensional vector space is a direct sum of Jordan cells, contradicting Proposition
\ref{contrex}.

\begin{proof}[Proof of Proposition \ref{decompJordan}]
We first define a binary relation $\calR$ on $\N$ as follows:
$i \calR j$ if and only if $u(e_j)$ is a non-zero scalar multiple of $e_i$.
Since every column (respectively, row) of the matrix of $u$ in $(e_n)_{n \in \N}$ has at most one non-zero entry,
we find that, for all $j \in \N$, there is at most one $i\in \N$ such that $i \calR j$ (respectively, $j \calR i$).

Next, we define $\sim$ as the sharpest equivalence relation that implies $\calR$, i.e.
given $(x,y)\in \N^2$, one has $x \sim y$ if and only
if there exists a chain $(a_0,\dots,a_n)$ of elements of $\N$ (possibly with $n=0$) such that
$x=a_0$, $y=a_n$ and, for all $i \in \lcro 0,n-1\rcro$, $a_i \calR a_{i+1}$ or $a_{i+1} \calR a_i$.

Given an equivalence class $X$ for $\sim$, we set $V_X:=\Span(e_k)_{k \in X}$. Hence
$V=\underset{X \in \N/\sim}{\bigoplus} V_{X.}$

Now, let us fix $X \in \N/\sim$. We will prove that $u$ stabilizes $X$ and the endomorphism of $V_X$
induced by $u$ is a Jordan cell: this will complete the proof.

Some preliminary work is required on the structure of $X$.
To start with, there is no sequence $(a_n)_{n \in \N}\in \N^\N$ such that $a_{n+1} \calR a_n$ for all $n \in \N$,
because for such a sequence we would have by induction that $u^n(e_{a_0})$ is a non-zero scalar multiple of $e_{a_n}$ for all
$n \in \N$, contradicting the fact that $u$ is locally nilpotent. In particular, in a chain
$(b_0,\dots,b_n)$ such that $b_i \calR b_{i+1}$ for all $i \in \lcro 0,n-1\rcro$, the $b_i$'s are pairwise distinct
(otherwise we would have a subchain $(b_i,\dots,b_{j-1},b_j)$ with $j>i$ and $b_i=b_j$, and the infinite
sequence $(b_j,b_{j-1},\dots,b_i,b_{j-1},\dots,b_i,\dots)$ obtained by repeating
the chain $(b_j,b_{j-1},\dots,b_{i+1})$ \emph{ad libitum} would satisfy the above assumption).

Now, because of the previous facts, there must be an element $a_0 \in X$ such that there is no $b \in X$ with
$b \calR a_0$. In particular, we have $u(e_{a_0})=0$.
If there is some $a_1 \in X$ such that $a_0 \calR a_1$, we continue, otherwise we stop. If in the first case there is some $a_2 \in X$ such that $a_1 \calR a_2$, then we continue, otherwise we stop.
And so on. This leads to two possibilities:
\begin{enumerate}[(i)]
\item Either we end up with a finite sequence $(a_0,\dots,a_p) \in X^{p+1}$ such that $a_i \calR a_{i+1}$ for all $i \in \lcro 0,p-1\rcro$, and no element $x \in X$ satisfies $x \calR a_0$ or $a_p \calR x$. \\
    In that case, using the first remark after the definition of $\calR$, we see that if an element $x \in X$
    satisfies either $x \calR a_i$ or $a_i \calR x$ for some $i \in \lcro 0,p\rcro$, then $x \in \{a_0,\dots,a_p\}$. Hence, by coming back to the definition of $\sim$, it follows that $X=\{a_0,\dots,a_p\}$. Moreover, we know from a previous remark that the
    $a_i$'s are pairwise distinct.
    By induction, we define a finite sequence $(f_0,\dots,f_p) \in \prod_{k=0}^p (\F e_{a_k} \setminus \{0\})$ such that $f_0=e_{a_0}$ and, for all
    $i \in \lcro 1,p\rcro$, $f_i$ is the non-zero vector of $\Span(e_{a_i})$ such that $u(f_i)=f_{i-1}$.
    Note that $(f_0,\dots,f_p)$ is a basis of $V_X$ and remember that $u(f_0)=0$. If follows that
    $u$ stabilizes $V_X$ and the induced endomorphism is a Jordan cell (of size $p+1$).

\item Or we have an infinite sequence $(a_n)_{n \in \N}$ such that $a_i \calR a_{i+1}$ for all $i \in \N$,
and no element $x \in X$ satisfies $x \calR a_0$. Like in case (i), the $a_i$'s are then pairwise distinct and
$X=\{a_n \mid n \in \N\}$. By induction, we construct a sequence $(f_n)_{n \in \N} \in \prod_{n \in \N} (\F e_{a_n} \setminus \{0\}) $ such that
$f_0=e_{a_0}$ and, for all $n \in \N \setminus \{0\}$, $f_n$ is the sole non-zero vector of $\Span(e_{a_n})$
such that $u(f_n)=f_{n-1}$. Note that $u(f_0)=0$. Hence, $(f_n)_{n \in \N}$ is a basis of $V_X$, $u$ stabilizes $V_X$, and the induced endomorphism is an infinite Jordan cell.
\end{enumerate}
This completes the proof.
\end{proof}

To finish, we give, in the formulation of \cite{Slowik}, an explicit counterexample that
corresponds to the example $u$ from the proof of Proposition \ref{contrex}.
Using the notation from that proof, we denote by $M\in \calM_{Cf}(\F)$ the matrix of $u$ in the basis
$(1,t,t^2,st^2,t^3,st^3,s^2t^3,\dots,t^n,st^{n},\dots,s^{n-1}t^n,\dots)$.
The matrix $M$ can be easily computed: its column number $0$ is zero, and for each
$j \in \N\setminus\{0\}$, its column number $j$ contains exactly one non-zero entry, which equals $1$
and is located on row number $0$ if $j=\frac{n(n+1)}{2}+1$ for some $n \in \N$,
and located on row number $j-1$ otherwise.
Interestingly, with the notable exception of the first one, each row of this matrix has at most one non-zero entry,
so it is very close to satisfying the conclusion of lemma 2.1 of \cite{Slowik}!
Note that $M$ is strictly upper-triangular. However, $u$ is not a direct sum of Jordan cells so we know from Proposition \ref{decompJordan} that it cannot be represented by a matrix in which every row and column contains at most one non-zero entry. Hence, $M$ is conjugated to no such matrix.

\end{document}